\documentclass[a4paper,11pt]{elsarticle}

\usepackage[utf8]{inputenc}
\usepackage[english]{babel}
\usepackage{enumerate}
\usepackage{latexsym}
\usepackage{amsthm,amsmath}
\usepackage{amssymb}
\usepackage{aliascnt} %para que funcione bien \Cref
\usepackage{multicol}
\usepackage{tikz}
\usepackage{float} % for [H] option in figures

\usepackage{hyperref}

\usetikzlibrary{babel,decorations.pathreplacing,calc,positioning,shapes.geometric}

\usepackage{cleveref}

\definecolor{10}{RGB}{115,59,171}
\definecolor{8}{RGB}{212,122,240}
\definecolor{7}{RGB}{99,212,119}
\definecolor{6}{RGB}{183,240,164}
\definecolor{D}{RGB}{255,162,79}
\definecolor{E}{RGB}{255,84,0}
\definecolor{F}{RGB}{158,248,255}
\definecolor{G}{RGB}{128,135,255}
\definecolor{I}{RGB}{187,255,0}
\definecolor{A}{cmyk}{.9,.05,.4,0}
\definecolor{B}{RGB}{150,30,150}
\definecolor{C}{RGB}{186,155,189}
\definecolor{9}{RGB}{0,180,60}
\definecolor{0}{RGB}{30,123,191}
\definecolor{1}{RGB}{255,113,102}
\definecolor{2}{RGB}{41,199,92}
\definecolor{3}{RGB}{242,207,16}
\definecolor{5}{RGB}{255,15,154}
\definecolor{4}{rgb}{.8,0,.8}

\definecolor{Red}{rgb}{1,0.4,0.4}
\definecolor{Green}{rgb}{.1,.5,.1}
\definecolor{Blue}{rgb}{.1,.1,.5}
\definecolor{blue}{RGB}{0,0,255}
\definecolor{Yellow}{rgb}{.8,.4,0}
\definecolor{X}{rgb}{.8,.4,0}
\definecolor{H}{rgb}{0,0,1}
\definecolor{light}{rgb}{.67,.84,.90}
\definecolor{Cyan}{rgb}{0,1,1}
\definecolor{Purple}{rgb}{.5,0,.5}
\definecolor{Purple2}{rgb}{.5,.2,.5}
\definecolor{white}{rgb}{1.0,1.0,1.0}
\definecolor{Purple2}{rgb}{.8,.4,0}
\definecolor{Amarillo}{RGB}{225,191,73}
\definecolor{Celeste}{RGB}{117,170,219}
\definecolor{Castano}{RGB}{232,53,17}
\definecolor{Black}{RGB}{0,0,0}
\definecolor{White}{RGB}{255,255,255}
\definecolor{gris}{rgb}{.5,.5,.5}

\newtheorem{theorem}{Theorem}[section]
\newaliascnt{corollary}{theorem}
\newtheorem{corollary}[corollary]{Corollary}
\aliascntresetthe{corollary}
\newaliascnt{lemma}{theorem}
\newtheorem{lemma}[lemma]{Lemma}
\aliascntresetthe{lemma}
\newaliascnt{proposition}{theorem}
\newtheorem{proposition}[proposition]{Proposition}
\aliascntresetthe{proposition}
\theoremstyle{definition}
\DeclareMathOperator{\im}{\mathcal{R}}
\DeclareMathOperator{\nulo}{\mathcal{N}}
\DeclareMathOperator{\unovec}{1\!\!1}
\DeclareMathOperator{\Supp}{Supp}
\DeclareMathOperator{\nul}{nl}
\DeclareMathOperator{\rank}{rk}
\DeclareMathOperator{\Int}{Int}
\DeclareMathOperator{\adj}{adj}

\newcommand{\cliqueker}[1]{\operatorname{Cker}{\left(#1\right)}}

\pgfdeclarelayer{background2}
\pgfdeclarelayer{background}
\pgfdeclarelayer{foreground}
\pgfsetlayers{background2,background,main,foreground}

\begin{document}
	
\begin{abstract}
	We study the nullspace of the adjacency matrix of split graphs, whose vertex set can be partitioned into a clique and an independent set. We introduce the \emph{clique-kernel}, a subspace that decides whether clique vertices lie in the support of a kernel eigenvector, and we prove that its dimension is at most one.  This yields the formula
	\[
	\nul(Sp) = \nul(R) + \dim(\cliqueker{Sp}),
	\]
	which fully describes the nullity of a split graph in terms of its biadjacency submatrix \(R\).
	We also analyze unbalanced split graphs through the concept of swing vertices and characterize the structure of their kernel supports. 
	Furthermore, we study the behavior of the nullspace under Tyshkevich composition and derive a closed formula for the determinant.
\end{abstract}

	%\end{}

	\begin{frontmatter}
		\author[daj,daj2]{Daniel A. Jaume}
		\ead{djaume@unsl.edu.ar}
		
		\author[daj2]{Victor N. Schv\"{o}llner}
		\ead{vnsi9m6@gmail.com}

		\author[daj]{Cristian Panelo}
		\ead{crpanelo@unsl.edu.ar}
		
		\author[daj,daj2]{Kevin Pereyra}
		\ead{kpereyra@unsl.edu.ar}
		
		\title{On the nullspace of split graphs}

		\address[daj]{Universidad Nacional de San Luis. Argentina}
		\address[daj2]{Instituto de Matem\'{a}ticas Aplicadas de San Luis -- CONICET. Argentina}
		
		%    \cortext[cor1]{Corresponding author: Daniel A. Jaume}

		%\date{Received: date / Accepted: date}
		% The correct dates will be entered by the editor
	\end{frontmatter}

	%%%%%%%%%%%%%%%%%%%%%%%%%%%%%%%%%%%%%%%%%%%%%%%%%%%%%%%
	
\section{Introduction}

The study of spectral properties of graphs, particularly the nullspace of the adjacency matrix, has emerged as a fundamental area in algebraic graph theory with significant implications for both theoretical mathematics and applied sciences. A graph $G$ is said to be \emph{singular} if its adjacency matrix $A(G)$ is singular, that is, if the dimension of the nullspace of $A(G)$ is at least one. The \emph{nullity} $\nul(G)$ of a graph $G$ is the multiplicity of zero as an eigenvalue of $A(G)$ or, equivalently, the dimension of the nullspace of $A(G)$. Understanding the structure of singular graphs and characterizing their nullspace has been a subject of intensive investigation since 1957 when von Collatz and Sinogowitz posed the problem, see \cite{von1957spektren}.

In her seminal work \cite{sciriha2007characterization} Sciriha introduced the notion of core vertices (those vertices whose entries in a kernel eigenvector are nonzero) and laid the foundations for the study of singular graphs. This approach led to a systematic characterization of singular graphs in terms of minimal configurations and core structures. Subsequently, numerous researchers have contributed to this framework, including studies on nut graphs (core graphs of nullity one) \cite{sciriha1998construction}, maximal core sizes \cite{sciriha2009maximal}, and the relationship between core and Fiedler vertices \cite{ali2016coalescing}. Jaume and Molina, see \cite{jaume2018null}, presented a decomposition of trees in terms of the supported vertices, which turn out to be equivalent to the classical Edmond-Gallai decomposition, see  \cite{jaume2021null}.

Graphs whose vertex set can be partitioned into a clique and an independent set were first singled out by F\"oldes and Hammer in 1977, 
(see \cite{FH77b} and \cite{FH77a}), under the name \emph{split graphs}. The same authors had earlier studied them implicitly as the graphs that do \emph{not} contain an induced copy of $2K_{2}$, $C_{4}$ or $C_{5}$ (\cite{FH76}).  
The term ``split'' was coined to stress that the vertex set is split into two sets whose induced subgraphs are as dissimilar as possible.  
Subsequent monographs of Golumbic \cite{golumbic2004algorithmic} and Mahadev--Peled \cite{mahadev1995threshold} embedded the class into the wider hierarchy of perfect graphs and made the notion standard in graph theory. Many NP-hard problems that are intractable for general graphs admit efficient algorithms when restricted to split graphs \cite{hammer1981splittance}. Moreover, split graphs can be recognized in linear time, and classical problems such as finding maximum cliques and maximum independent sets can be solved efficiently on this class \cite{foldes1977split}.

The spectral properties of split graphs have received considerable attention. Threshold graphs, which form a proper subclass of split graphs characterized by the absence of induced $P_4$, $C_4$, and $2K_2$ \cite{mahadev1995threshold}, have been extensively studied from a spectral perspective. Sciriha \cite{sciriha2011spectrum} investigated the spectrum of threshold graphs and established that singular threshold graphs possess kernel eigenvectors with at most two nonzero entries. Recent work by Panelo, Jaume, and Machado Toledo \cite{panelo5316214core} on the core-nilpotent decomposition of threshold graphs further illuminates the algebraic structure of these important graph families.

Despite the rich literature on both singular graphs and split graphs, a comprehensive study of the nullspace structure specifically for split graphs has been lacking. The block structure of the adjacency matrix of a split graph $Sp$,
\[
A(Sp) = \begin{pmatrix}
	J - I & R \\
	R^t & 0
\end{pmatrix},
\]
where $J$ is the all-ones matrix, $I$ is the identity matrix, and $R$ encodes adjacencies between the clique $K$ and independent set $S$, naturally suggests a systematic approach to understanding the nullspace of $A(Sp)$. This structure allows us to decompose the nullspace problem into more manageable components and to establish clear relationships between the rank properties of $R$ and the nullity of $Sp$.

In this paper, we provide a thorough investigation of the nullspace of split graphs, establishing several new results that characterize when a split graph is singular and describing the structure of its kernel eigenvectors. Our main contributions include:

\begin{enumerate}
	\item A complete characterization of the support of kernel eigenvectors (\Cref{sec:nullspace_split}), showing that for many split graphs, the support lies entirely within the independent set $S$. We establish necessary and sufficient conditions for when clique vertices can appear in the support.
	
	\item A detailed analysis of \emph{unbalanced} split graphs (\Cref{sec:unbalanced}), i.e., those whose s-partition is not unique. We prove that the nullspace structure depends critically on whether the set of swing vertices (those which can belong to either part of an s-partition) forms a clique or an independent set. 
	
	\item For split graphs with supported clique vertices (\Cref{sec:supp.clique}), we prove that the nullity equals $\nul(R) + 1$ when the support intersects the clique non-trivially, and we provide explicit formulas relating kernel eigenvectors to the adjugate matrix when $\nul(Sp) = 1$ (\Cref{sec:null.1}).
	
	\item Applications to the Tyshkevich composition (\Cref{sec:tyshk.null}), 
	%and the intersection graph construction (Section 8)
	demonstrating how our nullspace characterizations extend to composite structures.
	
	\item An explicit determinant formula for split graphs (\Cref{sec:det.split}) expressed in terms of the matrix $R$ using the Schur complement and Cauchy's formula, which facilitates computational verification of singularity.

\end{enumerate}

Our results extend and complement existing work on singular graphs by exploiting the special structure of split graphs. While Sciriha's general theory \cite{sciriha2007characterization} applies to arbitrary graphs, the block decomposition of split graph adjacency matrices enables us to obtain more precise characterizations. Furthermore, our treatment of unbalanced split graphs and swing vertices reveals phenomena that have not been systematically explored in the literature on singular graphs.

The remainder of this paper is organized as follows. \Cref{sec:notation} introduces notation and provides background on split graphs. \Cref{sec:nullspace_split} establishes general facts about the kernel of split graphs, including fundamental relationships between the nullspaces of $R$, $R^t$ and $A(Sp)$. \Cref{sec:unbalanced} analyzes the support of unbalanced split graphs and its relation with swing vertices. \Cref{sec:supp.clique} focuses on split graphs whose  support intersects the clique, and \Cref{sec:null.1} specializes in those with nullity 1. \Cref{sec:tyshk.null} is about the nullspace of a Tyshkevich composition. \Cref{sec:det.split} derives an explicit determinant formula. 
%Section~8 introduces the intersection graph construction and establishes conditions under which $\nul(Sp)\leq 1$.

%%%%%%%%%%%%%%%%%%%%%%%%%%%%%%%%%%%%%%%%%%%%%%%%%
%
%
%
%%%%%%%%%%%%%%%%%%%%%%%%%%%%%%%%%%%%%%%%%%%%%%%%%

%%%%%%%%%%%%%%%%%%%%%%%%%%%%%%%%%%%%%%%%%%%%%%%%%
%
%
%
%%%%%%%%%%%%%%%%%%%%%%%%%%%%%%%%%%%%%%%%%%%%%%%%%

\section{Notation}\label{sec:notation}

If $A$ is a matrix, then $A_{ij}$ denotes the $(i,j)$-entry of $A$ (i.e., the entry of $A$ in the intersection of row $i$ and column $j$), $A_{i*}$ denotes the row $i$ of $A$, and $A_{*j}$ denotes the column $j$ of $A$. If $A$ is square, we denote the adjugate of $A$ by $\adj(A)$. The symbol $e_i$ refers to the $i$\textit{-th canonical} vector, i.e., the vector with a ``1'' in the $i$-th position and zeros elsewhere. By $I$ we denote the identity matrix. Sometimes, we write a singleton set $\{x\}$ simply as ``$x$''.

We denote by $A(G)$ the adjacency matrix of $G$. The range of \(A\), i.e., the vector space spanned by the columns of \(A\), is denoted by \(\im(A)\). Its dimension, the rank of \(A\), is denoted by \(\rank(A)\). If $G$ is a graph, then $\im(G)$ and $\rank(G)$ denote, respectively, the range and the rank of its adjacency matrix. The nullspace of \(A\) is denoted by \(\nulo(A)\). Its dimension, the nullity of \(A\), is denoted by \(\nul(A)\). If $G$ is a graph, then $\nulo(G)$ and $\nul(G)$ denote, respectively, the nullspace and the nullity of its adjacency matrix. 

Let $U$ be a set of vectors of some vector space $V$; by $\langle U\rangle$ we denote the subspace spanned by all the linear combinations of vectors in $U$. We just write $\langle u\rangle$ instead of $\langle \{u\}\rangle$. $J$ is the ``all-ones'' matrix. Since $\rank(J)=1$, $\nul(J)$ is the number of columns of $J$ minus 1. $\unovec=(1,\ldots,1)^t$ is the ``all-ones'' vector (i.e., the matrix $J$ with just one column), where ``t'' means transpose. If $A$ is a nonzero matrix with non-negative entries, note that $\unovec\notin\nulo(A)$. Clearly, $\im(J)=\langle\unovec\rangle$. If $x=(x_1,\ldots,x_n)^t\in\mathbb{R}^n$, we have \(\unovec^t x=\sum_{i=1}^{n}x_i\). Clearly, $\unovec^t x=0$ if and only if $x\in\nulo(J)$.

If $G$ is a graph, then $N_G(v)=\{u\in V(G): uv\in E(G)\}$ is the neighborhood of the vertex $v$ in $G$. The degree in $G$ of a vertex $v$ (i.e., $|N_G(v)|$) is denoted by $\deg_G(v)$. The symbols $\omega(G)$ and $\alpha(G)$ denote, respectively, the clique number and the independence number of $G$. When $G$ is clear from the context, we can write $N(v), \deg(v), \omega$ and $\alpha$ instead of $N_G(v), \deg_G(v), \omega(G)$ and $\alpha(G)$, respectively. 

A graph $G\neq (\varnothing,\varnothing)$ is said to be \emph{split} if there exists a clique $K$ in $G$ and an independent set $S$ in $G$ such that $V(G)=K\dot{\cup}S$ (the symbol $\dot{\cup}$ denotes the disjoint union of two sets). In such a case, we use the notation $Sp$ instead of $G$ and we say that the pair $(K,S)$ is an \emph{s-partition} for $Sp$. Directly from the definition, it follows that every induced subgraph of a split graph is also split. We call \emph{clique vertices} the elements of $K$, and \emph{independent vertices} the elements of $S$. Two s-partitions $(K,S)$ and $(K',S')$ for $Sp$ are equal if $K=K'$ and $S=S'$.

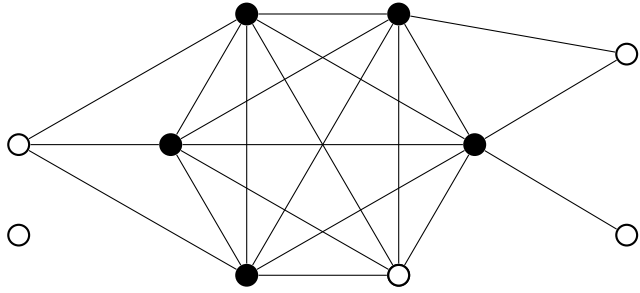
\begin{figure}[h]
	\centering
	\begin{tikzpicture}[scale=2, every node/.style={scale=1.3, circle, draw, thick, minimum size=6pt, inner sep=2pt}]
		% --- Definimos nodos de la clique (K6) en forma de hexágono ---
		\foreach \i in {0,...,5} {
			\node[fill=black] (k\i) at ({cos(60*\i)}, {sin(60*\i)}) {};
		}
		
		% --- Conectamos completamente la clique (K6) ---
		\foreach \i in {0,...,5} {
			\foreach \j in {\i,...,5} {
				\ifnum\i<\j
				\draw (k\i) -- (k\j);
				\fi
			}
		}
		
		% --- Vértices aislados (blancos) ---
		\node[fill=white] (s1) at (-2, 0) {};
		\node[fill=white] (s2) at (-2, -0.6) {};
		\node[fill=white] (s3) at (2, 0.6) {};
		\node[fill=white] (s4) at (2, -0.6) {};
		
		\draw (s3) -- (k1);
		\draw (s3) -- (k0);
		\draw (s4) -- (k0);
		\draw (s1) -- (k2);
		\draw (s1) -- (k3);
		\draw (s1) -- (k4);
		% --- No tienen aristas: son aislados ---
		
		%\onslide<6->{
			\node[fill=white] (k5b) at ({cos(60*5)}, {sin(60*5)}) {};
			%}
	\end{tikzpicture}
	\caption{Split graph (clique vertices in black, independent in white).}
	\label{ejemplo.split}
\end{figure}

If $Sp$ is a split graph, the notation $Sp(K,S)$ means that we are referring to $Sp$ together with an s-partition $(K,S)$ for it. Unless otherwise specified, we implicitly assume that both $K$ and $S$ are nonempty, and the adjacency matrix of $Sp$ has the form
\[
A(Sp)=
\begin{pmatrix}
	J-I & R \\
	R^t & 0
\end{pmatrix},
\] 
after applying the ordering $(K,S)$ to $V(Sp)$. Therefore, the blocks $J-I$ and $0$ of $Sp$ are square, of order $|K|$ and $|S|$ respectively, while $R$, the \emph{biadjacency} (i.e., bipartite-adjacency) matrix associated to \(Sp\), is of size $|K|\times |S|$.

%%%%%%%%%%%%%%%%%%%%%%%%%%%%%%%%%%%%%%%%%%%%%%%%%
%
%
%
%%%%%%%%%%%%%%%%%%%%%%%%%%%%%%%%%%%%%%%%%%%%%%%%%
%%%%%%%%%%%%%%%%%%%%%%%%%%%%%%%%%%%%%%%%%%%%%%%%%
%
%
%
%%%%%%%%%%%%%%%%%%%%%%%%%%%%%%%%%%%%%%%%%%%%%%%%%

\section{Nullspace of split graphs}\label{sec:nullspace_split}

	When studying split graphs one is naturally led to investigate the column space of the biadjacency matrix \(R\) and its interplay with the all-ones vector \(\unovec\). The next lemma establishes a general matrix-theoretic criterion that is pivotal for such analysis. %It characterizes the membership of a vector in the image of a matrix via the containment of nullspaces, a formulation that proves especially useful when applied to rank-one matrices like \(J = \unovec\unovec^t\).
\begin{lemma}
	\label[lemma]{lem_image_condition_rank1_perturbation}
	Let $A$ be any matrix and let $B = uv^t \neq 0$ for some vectors $u, v$. 
	If $u \neq 0$, then $v \in \im(A)$ if and only if $\nulo(A^t) \subset \nulo(B)$. 
	In particular: $\unovec \in \im(A)$ if and only if $\nulo(A^t) \subset \nulo(J)$.
\end{lemma}

\begin{proof}
	Since $u \neq 0$, we have $Bx = u(v^tx) = 0$ if and only if $v^tx = 0$. 
	Therefore, $\nulo(B) = \{x : v^tx = 0\} = v^{\perp}$.
	
	($\Rightarrow$) If $v \in \im(A)$, then $v = Aw$ for some vector $w$. 
	For any $x \in \nulo(A^t)$, we have $v^tx = w^t(A^tx) = w^t  0 = 0$, 
	so $x \in v^{\perp} = \nulo(B)$. Hence, $\nulo(A^t) \subset \nulo(B)$.
	
	($\Leftarrow$) Assume $\nulo(A^t) \subset v^{\perp}$. Then $v^tx = 0$ for all $x \in \nulo(A^t)$, 
	which means $v \in \nulo(A^t)^{\perp} = \im(A)$.
	
	The particular case follows by taking $u= v = \unovec$ and $B = J$.
\end{proof}

Let $A$ be any matrix with $n$ columns. The \emph{support} of $A$, denoted by $\Supp(A)$, is the set of all $v\in[n]=\{1,\ldots,n\}$ such that $x_v\neq 0$ for some vector $x\in\nulo(A)$. If $G$ is a graph, then the \emph{support} of $G$, denoted by $\Supp(G)$, is the support of its adjacency matrix, i.e., set of all vertices $v$ of $G$ such that $x_v\neq 0$ for some vector $x\in\nulo(G)$ (see \cite{jaume2018null}). 
Consequently, $V(G)\setminus\Supp(G)$ is the set of all vertices $v$ of $G$ such that $x_v=0$ for all $x\in\nulo(G)$. 
%We say that $G$ has an \emph{independent support} if $\Supp(G)$ is an independent set of $G$.

\begin{theorem}
	\label{kernel.split}
	If $Sp(K,S)$ is a split graph, then we have the following.
	\begin{enumerate}
		\item $(x_K,x_S)\in\nulo(Sp)$ if and only if $(I-J)x_K=Rx_S$ and $R^t x_K=0$.
		\item $\{\!(0,z)^t\! : z\in\nulo(R)\!\}\!\subset \!\nulo(Sp)$; equality holds if and only if \mbox{$\Supp(Sp) \!\subset\! S$.}
		\item $\nul(Sp)\! \geq \!\nul(R)\! = \!|S| \!-\! \rank(R)$; equality holds if and only if \mbox{$\Supp(Sp) \!\subset\! S$.}
		\item If $(x_K,x_S)\in\nulo(Sp)$ and $\unovec^t x_K=0$, then $x_K=0$.
		\item If $\unovec\in\im(R)$, then $\Supp(Sp) \subset S$.
	\end{enumerate}
\end{theorem}

\begin{proof}
	\begin{enumerate}[(1)]
		\item This follows immediately from the equation $A(Sp)(x_K,x_S)^t = 0$.
		
		\item If $z \in \nulo(R)$, then $A(Sp)(0,z)^t = (Rz, 0)^t = (0,0)^t$, so $(0,z)^t \in \nulo(Sp)$. 
		
		For the equality condition: If $\nulo(Sp) = \{(0,z)^t : z \in \nulo(R)\}$, then every kernel vector has its $K$-part zero, so $\Supp(Sp) \subset S$. Conversely, if $\Supp(Sp) \subset S$, then for any $(x_K, x_S) \in \nulo(Sp)$, we must have $x_K = 0$ (since any nonzero $x_K$ would put vertices from $K$ in the support), and by (1), $Rx_S = 0$, so $x_S \in \nulo(R)$.
		
		\item The inequality and the equality condition follow from part (2).
		
		\item From the hypothesis and part (1), we have $Rx_S = (I-J)x_K = x_K - Jx_K$. Since $\unovec^t x_K = 0$, we get $Jx_K = 0$, so $Rx_S = x_K$. Therefore, $x_K \in \im(R)$. Also from part (1), $R^t x_K = 0$, so $x_K \in \nulo(R^t)$. Hence, $x_K = 0$, because $\im(R)\cap\nulo(R^t)=\{0\}$ for any real matrix $R$.
		
		\item By \Cref{lem_image_condition_rank1_perturbation}, we have $\nulo(R^t) \subset \nulo(J)$. Now, for any $(x_K, x_S) \in \nulo(Sp)$, by (1) we have $R^t x_K = 0$, so $x_K \in \nulo(R^t) \subset \nulo(J)$, meaning $\unovec^t x_K = 0$. Then by part (4), $x_K = 0$. Therefore, $\Supp(Sp) \subset S$.\qedhere
	\end{enumerate}
\end{proof}

This theorem provides tools for understanding the nullspace structure of split graphs. In particular, item (5) gives a sufficient condition for when all support vertices lie in the independent set $S$.

It is important to note that the converse of item (5) of \Cref{kernel.split} is not true. In other words, there exist split graphs with $\Supp(Sp(K,S))\subset S$ and $\unovec\notin\im(R)$. To see that, consider the split graph $Sp(K,S)$ in \Cref{fig:counterexample.Supp.111111}. 
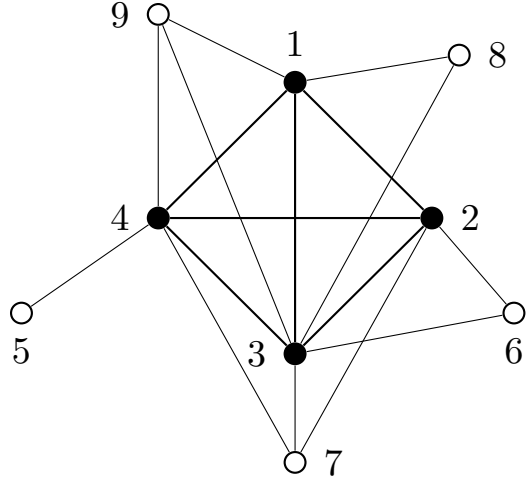
\begin{figure}[H]
	\centering
	\begin{tikzpicture}[scale=1.8, every node/.style={scale=1.3, circle, draw, thick, minimum size=6pt, inner sep=2pt}]
		
		% --- Clique K = {1,2,3,4} en forma de cuadrado ---
		\node[fill=black, label=above:$1$] (k1) at (0,1) {};
		\node[fill=black, label=right:$2$] (k2) at (1,0) {};
		\node[fill=black, label=left:$3$] (k3) at (0,-1) {};
		\node[fill=black, label=left:$4$] (k4) at (-1,0) {};
		
		% --- Conectamos completamente la clique ---
		\foreach \i/\j in {k1/k2, k1/k3, k1/k4, k2/k3, k2/k4, k3/k4} {
			\draw [thick](\i) -- (\j);
		}
		
		% --- Nodos independientes I = {5,...,9} ---
		\node[fill=white, label=below:$5$] (i5) at (-2,-0.7) {};
		\node[fill=white, label=below:$6$] (i6) at (1.6,-0.7) {};
		\node[fill=white, label=right:$7$] (i7) at (0,-1.8) {};
		\node[fill=white, label=right:$8$] (i8) at (1.2,1.2) {}; % ← movido para alejar {3,8}
		\node[fill=white, label=left:$9$] (i9) at (-1,1.5) {};
		
		% --- Conexiones según los vecindarios ---
		% N5 = {4}
		\draw (i5) -- (k4);
		
		% N6 = {2,3}
		\draw[bend right=10] (i6) -- (k2);
		\draw[bend left=10] (i6) -- (k3);
		
		% N7 = {2,3,4}
		\draw[bend left=10] (i7) -- (k2);
		\draw (i7) -- (k3);
		\draw[bend right=10] (i7) -- (k4);
		
		% N8 = {1,3}
		\draw[bend left=8] (i8) -- (k1);
		\draw[bend right=8] (i8) -- (k3);
		
		% N9 = {1,3,4}
		\draw[bend left=8] (i9) -- (k1);
		\draw (i9) -- (k3);
		\draw[bend right=8] (i9) -- (k4);
		
	\end{tikzpicture}
	\caption{Split graph $Sp(K,S)$ with $\Supp(Sp)\subset S$, but $\unovec\notin\im(R)$. }
	\label{fig:counterexample.Supp.111111}
\end{figure}
Direct computations show that $\Supp(Sp)=\{5,\ldots,9\}=S$. However, the equation $Rx=\unovec$ has no solution, otherwise we end up with the absurdity 
\[ 1=R_{*3}x=(x_6+x_7)+(x_8+x_9)=R_{2*}x+R_{1*}x=2. \]
In the next proposition, we generalize the argument used for the split graph in \Cref{fig:counterexample.Supp.111111}.

\begin{proposition}
	Let $Sp(K,S)$ be a split graph with a vertex $v\in K$ such that 
	\[ N_{Sp}(v)\cap S = \dot{\bigcup_{u\in W}} \big(N_{Sp}(u)\cap S\big), \] 
	for some subset $W\subset K$. Then, for any $y\in\im(R)$, we have 
	\[ y_v = \sum_{u\in W}y_u. \]
\end{proposition}

\begin{proof}
	Since $y\in\im(R)$, there exists $x\in\mathbb{R}^{|S|}$ such that $Rx=y$. Recall that for any vertex $w\in K$, the entry $y_w = (Rx)_w = R_{w*}x$ equals the sum of $x_i$ over all neighbors $i\in N(w)\cap S$.
	
	Using the disjoint union hypothesis:
	 \[ y_v\!=\!R_{v*}x\!=\!\sum_{i\in N(v)\cap S}x_i\!=\!\sum_{u\in W}\Big(\sum_{i\in N(u)\cap S}x_i\Big)\!=\!\sum_{u\in W}R_{u*}x\!=\!\sum_{u\in W}y_u. \qedhere\]
\end{proof}

The former proposition shows that the image of $R$ must satisfy certain linear constraints determined by the combinatorial structure of the split graph. In particular, if a vertex's neighborhood is the disjoint union of other neighborhoods, then their corresponding coordinates in any vector $y\in\im(R)$ are linearly dependent.\\

%%%%%%%%%%%%%%%%%%%%%%%%%%%%%%%%%%%%%%%%%%%%%%%%%
%
%
%
%%%%%%%%%%%%%%%%%%%%%%%%%%%%%%%%%%%%%%%%%%%%%%%%%

Let
\[
A(Sp) = 
\begin{pmatrix}
	J-I & R \\
	R^t & 0
\end{pmatrix}
\]
be the adjacency matrix of a split graph $Sp(K,S)$ such that $|K|\geq 2$. The \emph{clique-kernel} of \(Sp\) is defined as
\[ \cliqueker{Sp}:=\nulo(R^t)\cap\im((I-J)^{-1}R)\subset\mathbb{R}^{|K|}. \]

\begin{theorem}
	\label{kernel.split.basic.facts}
	Let $Sp(K,S)$ be a split graph such that $|K|\geq 2$.  
	\begin{enumerate}
		\item If $(x_K,x_S)^t\in\nulo(Sp)$, then $x_K\in \cliqueker{Sp}$.
		\item For each $x\in\cliqueker{Sp}$, there exists $y\in\mathbb{R}^{|S|}$ such that $(x,y)^t\in\nulo(Sp)$.
		\item $\Supp(Sp)\subset S$ if and only if $\cliqueker{Sp}=\{0\}$.
		\item If $Sp$ is singular, then $\Supp(Sp)\cap S\neq\varnothing$.
	\end{enumerate}
\end{theorem}

\begin{proof}
	\begin{enumerate}[(1)]
		\item By \Cref{kernel.split}(1), we have $(I-J)x_K=Rx_S$ and $R^t x_K=0$. Since $|K|\geq 2$, $I-J$ is invertible, so $x_K=(I-J)^{-1}Rx_S\in\im((I-J)^{-1}R)$. Combined with $x_K\in\nulo(R^t)$, this shows $x_K\in\cliqueker{Sp}$.
		
		\item If $x\in\cliqueker{Sp}$, then $R^t x=0$ and there exists $y\in\mathbb{R}^{|S|}$ such that $(I-J)x=Ry$. Therefore, $(x,y)^t\in\nulo(Sp)$ by \Cref{kernel.split}(1).
		
		\item ($\Rightarrow$) If $\Supp(Sp)\subset S$, then for any $(x_K,x_S)\in\nulo(Sp)$, we have $x_K=0$. In particular, for any $x\in\cliqueker{Sp}$, by (2) there exists $y$ such that $(x,y)\in\nulo(Sp)$, so $x=0$. Hence $\cliqueker{Sp}=\{0\}$.
		
		($\Leftarrow$) If $\cliqueker{Sp}=\{0\}$, then by (1), for any $(x_K,x_S)\in\nulo(Sp)$ we have $x_K=0$, so $\Supp(Sp)\subset S$.
		
		\item Suppose $Sp$ is singular, so there exists $(x_K,x_S)\in\nulo(Sp)$ with $(x_K,x_S)\neq 0$. If $x_S=0$, then by \Cref{kernel.split}(1), $(I-J)x_K=0$. Since $I-J$ is invertible for $|K|\geq 2$, this implies $x_K=0$, contradicting $(x_K,x_S)\neq 0$. Therefore, $x_S\neq 0$, so $\Supp(Sp)\cap S\neq\varnothing$. \qedhere
	\end{enumerate}
\end{proof}

The clique-kernel of $Sp$ captures precisely the part of the nullspace that involves the clique vertices. When this subspace is trivial, all the supported vertices of $Sp$ are independent.

%%%%%%%%%%%%%%%%%%%%%%%%%%%%%%%%%%%%%%%%%%%%%%%%%
%
%
%
%%%%%%%%%%%%%%%%%%%%%%%%%%%%%%%%%%%%%%%%%%%%%%%%%
The relationship between the sizes of the clique and independent set in a split graph has implications for its singularity. The following result establishes simple  combinatorial criteria that guarantee when a split graph must be singular based solely on the cardinalities of its s-partition. In particular, we show that an imbalance in favor of the independent set always forces singularity, while in the balanced case, the singularity of the entire graph reduces to the singularity of the biadjacency matrix. %This provides a bridge between the graph's global spectral properties and the local structure encoded in the matrix $R$.
\begin{theorem}
	\label{|K|<|I|.singular}
	Let $Sp(K,S)$ be a split graph. 
	\begin{enumerate}
		\item If $|K|<|S|$, then $Sp$ is singular.
		\item If $|K|=|S|$, then $Sp$ is singular if and only if $R$ is singular.
	\end{enumerate}
\end{theorem}

\begin{proof}
	\begin{enumerate}[(1)]
		\item Since $\rank(R)\leq\min\{|K|,|S|\}$, we have $\nul(R)=|S|-\rank(R)\geq |S|-|K|>0$. 
		By \Cref{kernel.split}(3), $\nul(Sp)\geq\nul(R)>0$, so $Sp$ is singular.
		
		\item ($\Rightarrow$) If $Sp$ is singular, then $\nul(Sp)>0$. 
		By \Cref{kernel.split}(3), we have $\nul(Sp)\geq\nul(R)$. 
		If $\nul(R)=0$, then any $(x_K,x_S)\in\nulo(Sp)$ must satisfy $x_K=0$ (otherwise $x_K$ would be a nonzero vector in $\nulo(R^t)$, but $\nulo(R^t)=\{0\}$ since $\nul(R)=0$ and $R$ is a square matrix). 
		Then from $(I-J)x_K=Rx_S$ we get $Rx_S=0$, so $x_S\in\nulo(R)=\{0\}$. 
		Thus $(x_K,x_S)=0$, contradicting $\nul(Sp)>0$. 
		Therefore $\nul(R)>0$, so $R$ is singular.
		
		($\Leftarrow$) If $R$ is singular, then $\nul(R)>0$. 
		By \Cref{kernel.split}(2), $\{(0,z)^t:z\in\nulo(R)\}\subset\nulo(Sp)$, so $\nul(Sp)\geq\nul(R)>0$. 
		Hence $Sp$ is singular. \qedhere
	\end{enumerate}
\end{proof}

%%%%%%%%%%%%%%%%%%%%%%%%%%%%%%%%%%%%%%%%%%%%%%%%%
%
%
%
%%%%%%%%%%%%%%%%%%%%%%%%%%%%%%%%%%%%%%%%%%%%%%%%%
Regularity conditions within the clique or independent set impose strong constraints on the nullspace structure of a split graph. This theorem reveals how degree homogeneity among clique vertices confines the support entirely to the independent set, while degree homogeneity among independent vertices forces a specific linear dependence in the kernel of $R$. %These results demonstrate how local symmetry conditions—all vertices in one part having identical degrees—propagate to global spectral properties, offering combinatorial tools to control the location of support vertices and the structure of the nullspace.
\begin{theorem}
	\label{same.K.S-deg}
	Let $Sp(K,S)$ be a connected split graph. 
	\begin{enumerate}
		\item If $\deg_{Sp}(x) = \deg_{Sp}(y)$ for all $x,y \in K$, then $\Supp(Sp) \subset S$.
		\item If $\deg_{Sp}(x) = \deg_{Sp}(y)$ for all $x,y \in S$, then $\nulo(R) \subset \nulo(J)$.
	\end{enumerate}
\end{theorem}

\begin{proof}
	\begin{enumerate}[(1)]
		\item The condition $\deg(x) = \deg(y)$ for all $x,y \in K$ implies that all vertices in $K$ have the same number of neighbors in $S$, since within $K$ they are all mutually adjacent (each has degree $|K|-1$ within $K$). Thus, there exists $d \geq 0$ such that $|N(x) \cap S| = d$ for all $x \in K$. 
		
		Since $Sp$ is connected and $S$ is nonempty (as $Sp$ is a split graph with $|K| \geq 1$ and connected), we must have $d \geq 1$. Therefore, each row sum of $R$ equals $d$, so:
		\[ \sum_{j=1}^{|S|} R_{*j} = d\unovec, \]
		which shows that $\unovec \in \im(R)$. The result follows from \Cref{kernel.split}(5).
		
		\item The condition $\deg(x) = \deg(y)$ for all $x,y \in S$ means that all vertices in $S$ have the same number of neighbors in $K$ (since vertices in $S$ have no neighbors in $S$). Thus, there exists $d \geq 0$ such that $|N(x) \cap K| = d$ for all $x \in S$.
		
		By connectedness and the fact that $K$ is nonempty, we have $d \geq 1$. Therefore, each column sum of $R$ equals $d$, which means each row sum of $R^t$ equals $d$, so:
		\[ \sum_{j=1}^{|K|} R^t_{*j} = d\unovec, \]
		which shows that $\unovec \in \im(R^t)$. By \Cref{lem_image_condition_rank1_perturbation}, this implies $\nulo(R) \subset \nulo(J)$. \qedhere
	\end{enumerate}
\end{proof}

%%%%%%%%%%%%%%%%%%%%%%%%%%%%%%%%%%%%%%%%%%%%%%%%%
%
%
%
%%%%%%%%%%%%%%%%%%%%%%%%%%%%%%%%%%%%%%%%%%%%%%%%%
The following corollary provides a sufficient combinatorial condition for nonsingularity that arises from a perfect ``neighborhood covering'' of the clique by the independent set.
\begin{theorem}
	Let $Sp(K,S)$ be a split graph. If $K = \mathop{\dot{\bigcup}}\limits_{v \in S} N_{Sp}(v)$, then $Sp$ is nonsingular.
\end{theorem}

\begin{proof}
	The condition $K = \dot{\bigcup}_{v \in S} N(v)$ means that the neighborhood sets $\{N(v): v \in S\}$ form a partition of $K$. This implies that each vertex in $K$ has exactly one neighbor in $S$, so every row of $R$ contains exactly one nonzero entry (which is 1). Consequently, all vertices in $K$ have the same degree in $Sp$, since each has degree $|K|-1$ within $K$ plus exactly 1 neighbor in $S$, giving $\deg(u) = |K|$ for all $u \in K$. By \Cref{same.K.S-deg}(1), this implies $\Supp(Sp) \subset S$. Then by \Cref{kernel.split.basic.facts}(3), we have $\cliqueker{Sp} = \{0\}$, and by \Cref{kernel.split}(3), $\nul(Sp) = \nul(R)$.
	
	Now, consider an arbitrary vector $x \in \mathbb{R}^{|S|}$ and suppose $Rx = 0$. Since each row of $R$ has exactly one nonzero entry, the equation $R_{i*}x = 0$ implies that for each $i \in K$, if $v \in S$ is the unique neighbor of $i$, then $x_v = 0$. As the neighborhoods cover all of $K$ and are disjoint, this means every $v \in S$ appears as the unique neighbor of some $i \in K$, so $x_v = 0$ for all $v \in S$. Therefore, $x = 0$ and $\nul(R) = 0$.
	
	Hence, $\nul(Sp) = \nul(R) = 0$, so $Sp$ is nonsingular.
\end{proof}

%%%%%%%%%%%%%%%%%%%%%%%%%%%%%%%%%%%%%%%%%%%%%%%%%%%%%%%%%%%
%%%%%%%%%%%%%%%%%%%%%%%%%%%%%%%%%%%%%%%%%%%%%%%%%
%
%
%
%%%%%%%%%%%%%%%%%%%%%%%%%%%%%%%%%%%%%%%%%%%%%%%%%
%%%%%%%%%%%%%%%%%%%%%%%%%%%%%%%%%%%%%%%%%%%%%%%%%
%
%
%
%%%%%%%%%%%%%%%%%%%%%%%%%%%%%%%%%%%%%%%%%%%%%%%%%
\section{Unbalanced split graphs}\label{sec:unbalanced}

A split graph $Sp(K,S)$ is said to be \emph{balanced} if $|K|=\omega(Sp)$ and $|S|=\alpha(Sp)$. Otherwise, we say that $Sp$ is \emph{unbalanced} (or not balanced). 
%The next theorem follows from the work of Hammer and Simeone [XXX] and appears in [XXX].
It is known that (see \cite{nordhausgaddumcheng2016split}, \cite{hammer1981splittance} and \cite{golumbic2004algorithmic}):
\begin{enumerate}
	\item $Sp$ is balanced if and only if it has a unique s-partition (thus, unbalanced split graphs are exactly those with two or more s-partitions);
	
	\item $Sp$ is balanced if and only if $\omega(Sp)+\alpha(Sp)=|Sp|$.
	
	\item $Sp(K,S)$ is unbalanced if and only if 
	\[ (|K|,|S|)\in\{(\omega(Sp),\alpha(Sp)-1),(\omega(Sp)-1,\alpha(Sp))\}. \]
	
	\item $Sp$ is unbalanced if and only if $\omega(Sp)+\alpha(Sp)=|Sp|+1$.
\end{enumerate}

%Then, by item (6) of \Cref{omega+alpha.split} we see that $Sp$ is balanced if and only if it has a unique s-partition. . From items (4) and (6) of \Cref{omega+alpha.split} we can also deduce that 

We say that a vertex $w$ of a split graph $Sp$ is \emph{swing} in $Sp$ if there exist s-partitions $(K,S),(K',S')$ for $Sp$ such that $w\in K$ and $w\in S'$. Let $W(Sp)$ be the set of all swing vertices of $Sp$. Clearly, $W(Sp)=W(\overline{Sp})$, and $Sp$ is balanced if and only if $W(Sp)=\varnothing$. However, $Sp\setminus W(Sp)$ may not be balanced in general (take $Sp\approx P_3$). 
%From item (4) of \Cref{omega+alpha.split} we deduce that 
%It is easy to see that $w\in W(Sp)$ if and only if there are s-partitions $(K,S),(K',S')$ for $Sp$ such that $w\in K$, $w\in S'$ and
%\begin{equation}
%	\label{swing_def}
%	K\setminus \{w\}=N_{Sp}(w)=K',
%\end{equation}
%\[ (|K|,|S|)=(\omega(Sp),\alpha(Sp)-1), \ (|K'|,|S'|)=(\omega(Sp)-1,\alpha(Sp)). \]
If $Sp$ is a split graph, let:
\begin{enumerate}
	%\item $W(Sp)$ be the set of swing vertices of $Sp$;
	
	\item $\mathcal{K}(Sp)$ be the set of all cliques $K$ of $Sp$ such that $(K,S)$ is an s-partition for $Sp$, for some $S$;
	
	\item $\mathcal{S}(Sp)$ be the set of all independent sets $S$ of $Sp$ such that $(K,S)$ is an s-partition for $Sp$, for some $K$;
	
	\item $K^*(Sp)$ be the set of vertices that are ``always clique'' in $Sp$, i.e., $K^*(Sp)=\bigcap (K\in \mathcal{K}(Sp))$;
	
	\item $S^*(Sp)$ be the set of vertices that are ``always independent'' in $Sp$, i.e., $S^*(Sp)=\bigcap (S\in \mathcal{S}(Sp))$. 
\end{enumerate}

Then, by the above definitions, it is clear that
\begin{equation}
	\label{vertex_tripartition}
	V(Sp) = W(Sp) \ \dot{\cup} \ K^*(Sp) \ \dot{\cup} \ S^*(Sp).
\end{equation}

The concept of swing vertex lies implicitly among several works on split graphs. But, as far as we know, the term ``swing'' first appears explicitly in \cite{whitman2020split} (page 8). However, there the definition of swing vertex is slightly different: a vertex $w$ is swing in $Sp(K,S)$ if 
\begin{equation}
	\label{whitman_def_swing}
	N_{Sp}(w)=K\setminus \{w\}.
\end{equation}
There is an issue with \eqref{whitman_def_swing}: it depends on the s-partition. To see that, consider the split graph $Sp=abc\approx P_3$ with the s-partitions $(K,S)=(\{a,b\},c)$ and $(K',S')=(\{b,c\},a)$. By our definition, $W(Sp)=\{a,c\}$, which is an invariant of $Sp$. However, using \eqref{whitman_def_swing}, we observe that: 1) in $Sp(K,S)$, $a$ is swing but $c$ is not; 2) in $Sp(K',S')$, $c$ is swing but $a$ is not. For this reason, we decided to modify \eqref{whitman_def_swing}. As we have seen above, this choice allows us to immediately obtain the useful vertex-decomposition \eqref{vertex_tripartition}.\\

A graph $G$ is called an \textit{NG-graph} if $\chi(G)+\chi(\overline{G})=|G|+1$, where $\chi(G)$ is the chromatic number of $G$. In particular, a split graph is NG if and only if it is unbalanced, because $\chi(G)=\omega(G)$ when $G$ is split, and $\alpha(Sp)+\omega(Sp)=|Sp|+1$ when $Sp$ is unbalanced. In \cite{nordhausgaddumcheng2016split} the authors study some structural and combinatorial properties of unbalanced split graphs in terms of what they call the $ABC$-partition of $G$, where:
\[ A_G=\{v\in V(G):\deg_G(v)=\chi(G)-1\}, \]
\[ B_G=\{v\in V(G):\deg_G(v)>\chi(G)-1\}, \]
\[ C_G=\{v\in V(G):\deg_G(v)<\chi(G)-1\}. \]
The ``$G$'' sub-index will be omitted if $G$ is clear from the context.

As we will see soon in \Cref{ABC=WK*S*}, $ABC$-partition and $WK^*S^*$-partition \eqref{vertex_tripartition} are essentially the same vertex decomposition. This will take several steps: \Cref{A_subset_W,B_subset_K*,C_subset_S*}. As a consequence of \Cref{ABC=WK*S*}, we obtain a translation (\Cref{cheng_collins_trenk}), in terms of $WK^*S^*$-partition, of Theorem 10 in \cite{nordhausgaddumcheng2016split}, which characterizes the s-partitions of an unbalanced split graph in terms of its $ABC$-partition. Moreover, thanks to \Cref{unb.max.clique.ind}, we obtain a slightly stronger version of the cited result. \Cref{cheng_collins_trenk} is fundamental to subsequently understand the support of unbalanced split graphs.

\begin{lemma}
	\label{A_subset_W}
	If $Sp(K,S)$ is unbalanced, then $A_{Sp}\subset W(Sp)$. 
\end{lemma}

\begin{proof}
	First of all, recall that $(|K|,|S|)\in\{(\omega,\alpha-1),(\omega-1,\alpha)\}$, since $Sp$ is unbalanced. Let $x\in A$. We have two cases: 1) $x\in K$; 2) $x\in S$.
	\begin{enumerate}[(1).]
		\item If $x\in K$, then $K\setminus x\subset N(x)$. 
		
		If $|K|=\omega-1$, then $N(x)=(K\setminus x)\dot{\cup}y$, for some $y\in S$. Since $N(y)\subset K$, we have $\deg(y)\leq\omega-1$. If $\deg(y)<\omega-1$, then $y\in C$, contradicting that there are no edges between $A$ and $C$ (see item 5 of Theorem 2 in \cite{nordhausgaddumcheng2016split}. So, $\deg(y)=\omega-1$, and hence $N(y)=K$. Therefore, $x\in W$ because $(N(x),(S\setminus y)\dot{\cup} x)$ is an s-partition for $Sp$ where $x$ is independent. 
		
		If $|K|=\omega$, then $N(x)=K\setminus x$, which means that $x$ has no neighbors in $S$. Now, it is easy to check that $(N(x), S\dot{\cup}x)$ is an s-partition for $Sp$ where $x$ is independent. 
		
		\item If $x\in S$, then $x$ is a clique vertex in $\overline{Sp}(S,K)$ with
		\[ \deg_{\overline{Sp}}(x)=|Sp|-1-\deg(x) = |Sp|-\omega = \alpha-1 = \omega(\overline{Sp})-1\]
		 (recall that $\alpha+\omega=|Sp|+1$, because $Sp$ is unbalanced), which means that $x\in A_{\overline{Sp}}$. Then, we can apply (1) to infer that $x\in W(\overline{Sp})=W$. \qedhere
	\end{enumerate}
\end{proof}

\begin{lemma}
	\label{B_subset_K*}
	If $Sp$ is unbalanced, then $B_{Sp}\subset K^*(Sp)$.
\end{lemma}

\begin{proof}
	Assume that $x\in B$ but there exists an s-partition $(K,S)$ for $Sp$ such that $x\in S$. Since $N(x)\subset K$ and $\deg(x)>\omega-1$, we have $\deg(x)=\omega$. Thus, $N(x)=K$ and $|K|=\omega$. Now, observe that $K\dot{\cup} x$ is a clique of size $\omega+1$ in $Sp$, which is absurd.
\end{proof}

\begin{lemma}
	\label{C_subset_S*}
	If $Sp$ is unbalanced, then $C_{Sp}\subset S^*(Sp)$.
\end{lemma}

\begin{proof}
	Assume that $x\in C$ but there exists an s-partition $(K,S)$ for $Sp$ such that $x\in K$. Since $K-x\subset N(x)$ and $\deg(x)<\omega-1$, we have $\deg(x)=\omega-2$. Thus, $N(x)=K\setminus x$ and $|K|=\omega-1$. Now, observe that $S\dot{\cup} x$ is an independent set of size $\alpha+1$ in $Sp$, which is absurd.
\end{proof}

Two vertices $u,v$ in a graph $G$ are said to be \emph{twins} in $G$ if
\[ N_G(u)\setminus v=N_G(v)\setminus u. \]
It is easy to check that ``to be twins'' is an equivalence relation over $V(G)$, and the equivalence classes are cliques or independent sets of $G$.

\begin{theorem}
	\label{ABC=WK*S*}
	If $Sp$ is an unbalanced split graph, then
	\[ A_{Sp}=W(Sp), \quad B_{Sp}=K^*(Sp), \quad C_{Sp}=S^*(Sp). \]
	Moreover, $W(Sp)$ is always a clique or an independent set, of mutually twin vertices.
\end{theorem}

\begin{proof}
	Since $A\dot{\cup}B\dot{\cup}C=V(Sp)=W\dot{\cup}K^*\dot{\cup}S^*$, the required identities follows from \Cref{A_subset_W,B_subset_K*,C_subset_S*}. The statement about $W$ can be subsequently derived from Theorem 2 in \cite{nordhausgaddumcheng2016split}, using that $A=W$.  
\end{proof}

\begin{corollary}
	Let $Sp(K,S)$ be a connected split graph such that $K=\bigcup_{v\in S}N_{Sp}(v)$ and $\rank(R)=1$. Then:
	\begin{enumerate}
		\item $R=J$;
		
		\item all the independent (clique) vertices of $Sp$ are pairwise twins;
		
		\item $(K,S)=(K^*(Sp),W(Sp))$ and $(|K|,|S|)=(\omega(Sp)-1,\alpha(Sp))$.
	\end{enumerate} 
\end{corollary}

\begin{proof}
	\begin{enumerate}[(1).]
		\item Since $Sp$ is connected and has no isolated vertices in $S$, all the columns of $R$ are nonzero. Since $K=\bigcup_{v\in S}N(v)$, all the rows of $R$ are nonzero. Moreover, since $\rank(R)=1$, all the columns of $R$ must be scalar multiples of each other. Since all columns are nonzero $\{0,1\}$-vectors, they must all be equal to the same vector. Since $R_{i*}\neq 0$ for all $i$ and $R$ is a $\{0,1\}$-matrix, we get that $R_{*j}=\unovec$ for all $j$.
		
		\item Since $R=J$, we have $N(v)=K$ for all $v\in S$ and $N(u)=(K\setminus u)\cup S$ for all $u\in K$. 
		%Thus, all the independent (clique) vertices of $Sp$ are pairwise twins.
		%every vertex in $S$ is adjacent to all vertices in $K$, i.e., $N(v)=K$ for all $v\in S$. Therefore, all the independent vertices are pairwise twins. Similarly, every vertex in $K$ is adjacent to all vertices in $S$ and all other vertices in $K$. Hence, $N(u)\setminus v= (K\setminus \{u,v\})\cup S=N(v)\setminus u$ for all $u,v\in K$, which means that all the clique vertices are pairwise twins.
		
		\item If $|K|=\omega$ and $v\in S$, then $K\cup v$ would be a clique of size $\omega+1$ in $Sp$, which is absurd. Then, $\deg(v)=\omega-1$, and hence $S\subset A=W$ by \Cref{ABC=WK*S*}. On the other hand, since every clique vertex of $Sp$ is universal, we have $K\subset B=K^*$ by \Cref{ABC=WK*S*}. \qedhere
	\end{enumerate}
\end{proof}

%The following proposition is a refinement of \Cref{cheng_collins_trenk}, because, essentially, tells us that ``$\mathcal{K}(Sp)$'' and ``$\mathcal{S}(Sp)$'' can be replaced with ``$Sp$'' in \Cref{cheng_collins_trenk}. 

\begin{proposition}\label{unb.max.clique.ind}
	Let $Sp$ be an unbalanced split graph. Then,
	\begin{enumerate}
		\item $\mathcal{K}(Sp)$ contains all the maximum cliques of $Sp$.
		
		\item $\mathcal{S}(Sp)$ contains all the maximum independent sets of $Sp$.
	\end{enumerate}
\end{proposition}

\begin{proof}
	\begin{enumerate}[(1).]
		\item Let $K$ be a maximum clique of $Sp$. Since $Sp$ is unbalanced, we know that $\alpha+\omega=|Sp|+1$, and so $\alpha-1=|Sp|-\omega=|V(Sp)\setminus K|$. Let $S$ be a maximum independent set of $Sp$.  Since $|K\cap S|\leq 1, |S|=\alpha$ and $|V(Sp)\setminus K|=\alpha-1$, it must be $S=(V(Sp)\setminus K)\dot{\cup} w$, for some $w\in K$. Consequently, $V(Sp)\setminus K$ is an independent set of $Sp$. Thus, $(K,V(Sp)\setminus K)$ is an s-partition for $Sp$, which means that $K\in\mathcal{K}(Sp)$.
		
		\item Let $S$ be a maximum independent set of $Sp$. Then, $S$ is a maximum clique of the graph $\overline{Sp}$ (the complement of $Sp$), which is clearly split and unbalanced. From (1), we know that $S\in\mathcal{K}(\overline{Sp})$, which means that $V(\overline{Sp})\setminus S$ is an independent set of $\overline{Sp}$. Hence, $V(Sp)\setminus S$ is a clique of $Sp$, which shows that $S\in\mathcal{S}(Sp)$. \qedhere
	\end{enumerate}
\end{proof}

\Cref{unb.max.clique.ind} does not hold in general for balanced split graph. In fact, every balanced $Sp(K,S)$ with an independent vertex of degree $|K|-1$ is a counterexample ($P_4$ is the ``simplest'' one). In other words, a balanced split graph $Sp(K,S)$ has always a unique s-partition, but $K$ might not be the only maximum clique of $Sp$. An analogous phenomenon can be observed for maximum independent sets. The following result clarifies this behavior of the balanced split graphs.

\begin{proposition}
	Let $Sp(K,S)$ be a balanced split graph. Then, we have the following.
	\begin{enumerate}
		\item $K$ is the only maximum clique of $Sp$ if and only if $Sp$ does not contain an independent vertex of degree $|K|-1$.
		
		\item $S$ is the only maximum independent set of $Sp$ if and only if $Sp$ does not contain a clique vertex $x$ such that $|N_{Sp}(x)\cap S|=1$.
	\end{enumerate}
\end{proposition}

\begin{proof}
	\begin{enumerate}[(1).]
		\item If $x$ is a vertex of $Sp$ of degree $|K|-1$, then $x$ is an independent vertex of \(G\) (otherwise, $x$ would be swing). Now, observe that $N(x)\cup x$ is a maximum clique of $Sp$, different from $K$.
		
		For the converse, suppose that $K'$ is a maximum clique of $Sp$, different from $K$. Since $|K'\cap S|\leq 1$, it must be $K'=(K\setminus w)\dot{\cup}x$, for some $w\in K$ and $x\in S$. Since $x\in K'$, we have $K\setminus w\subset N(x)$. Since $x\in S$, $x$ is not adjacent to any vertex of $S$, which means that $N(x)\subset K$. If $w\in N(x)$, then $N(x)=K$ and hence $x\in W(Sp)=\varnothing$. This absurdity implies $N(x)=K\setminus w$, i.e., $\deg(x)=|K|-1$. 
		
		\item If $x$ is a vertex of $Sp$ such that $N(x)\cap S=\{w\}$ for some $w$, it is clear that $x\in K$ and $(I\setminus w)\dot{\cup}x$ is a maximum independent set different from $S$.
		
		For the converse, suppose that $S'$ is a maximum independent set of $Sp$, different from $S$. Since $|K\cap S'|\leq 1$, it must be $S'=(S\setminus w)\dot{\cup}x$, for some $w\in S$ and $x\in K$. Since $x\in K$, $K\setminus x\subset N(x)$. Since $x\in S'$, $x$ is not adjacent to any vertex of $S\setminus w$. Notice that $w\in N(x)$ (otherwise, $x$ would be swing). Therefore, $N(x)=(K\setminus x)\dot{\cup} w$, which means that $|N(x)\cap S|=1$. \qedhere
	\end{enumerate}
\end{proof}

\begin{corollary}
	\label{cheng_collins_trenk}
	Let $Sp$ be an unbalanced split graph and let $\mathcal{KS}$ be the collection of all the s-partitions for $Sp$. Then, we have the following (for simplicity, we will use $W,K^*,S^*$ instead of $W(Sp),K^*(Sp),S^*(Sp)$).
	\begin{enumerate}
		\item If $W=\{w\}$, then 
		\[ \mathcal{KS}=\{ (K^*\cup w, S^*), (K^*, S^*\cup w)  \}; \]
		moreover, $K^*\cup w$ is the unique maximum clique of $Sp$ and $S^*\cup w$ is the unique maximum independent set of $Sp$. 
		
		\item If $W$ is a clique of size $\geq 2$, then 
		\[ \mathcal{KS}=\{ (K^*\cup W, S^*)\}\cup\{ (K^*\cup(W\setminus w), S^*\cup w):w\in W \}; \]
		moreover, $K^*\cup W$ is the unique maximum clique of $Sp$, and all the maximum independent sets of $Sp$ are of the form $S^*\cup w$.
		
		\item If $W$ is an independent set of size $\geq 2$, then
		\[ \mathcal{KS}=\{ (K^*, S^*\cup W)\}\cup\{ (K^*\cup w, S^*\cup(W\setminus w)):w\in W \}; \]
		moreover, $S^*\cup W$ is the unique maximum independent set of $Sp$, and all the maximum cliques of $Sp$ are of the form $K^*\cup w$.
	\end{enumerate}
\end{corollary}

\begin{proof}
	It follows from \Cref{ABC=WK*S*}, Theorem 10 in \cite{nordhausgaddumcheng2016split} and \Cref{unb.max.clique.ind}.
\end{proof}

Notice that \Cref{cheng_collins_trenk} implicitly characterizes a swing vertex in terms of its neighborhood (compare with \eqref{whitman_def_swing}): 
\begin{enumerate}[(1).]
	\item If $W$ is a clique, then $w$ is swing if and only if $N(w)=K^*\cup(W\setminus w)=\{v\neq w:\deg(v)\geq\omega-1\}$;
	
	\item If $W$ is an independent set, then $w$ is swing if and only if $N(w)=K^*=\{v:\deg(v)\geq\omega\}$.
\end{enumerate}

%%%%%%%%%%%%%%%%%%%%%%%%%%%%%%%%%%%%%%%%%%%%%%%%%

%%%%%%%%%%%%%%%%%%%%%%%%%%%%%%%%%%%%%%%%%%%%%%%%%
%
%
%
%%%%%%%%%%%%%%%%%%%%%%%%%%%%%%%%%%%%%%%%%%%%%%%%%

%If $w\in I$, then $w$ is swing if and only if $\deg_S(w) =|K|$. By converting $w$ into a clique vertex, we obtain the split graph $(S,K\dot{\cup}w,I-w)$. If $w\in K$, then $w$ is swing if and only if $\deg_S(w) =|K|-1$. By converting $w$ into an independent vertex, we obtain the split graph $(S,K-w,I\dot{\cup}w,)$. 
 The following theorem states that in unbalanced split graphs, supported vertices must lie in the intersection of all the maximum independent sets. %This has significant implications for understanding the nullspace structure of unbalanced split graphs.
 
\begin{theorem}
	\label{unbalanced.supp}
	Let $Sp$ be an unbalanced split graph. Then, $\Supp(Sp)$ is contained in the intersection of all the maximum independent sets of $Sp$.
\end{theorem}

\begin{proof}
	Consider an arbitrary s-partition $(K,S)$ for $Sp$ such that $|S|=\alpha(Sp)$. The existence of such is guaranteed by \Cref{cheng_collins_trenk}. Then, by \Cref{cheng_collins_trenk}, there exists a vertex $w\in S$ such that $N(w)=K$, which means that $R_{*w}=\unovec$. Since $\unovec\in\im(R)$, we conclude by \Cref{kernel.split}(5) that $\Supp(Sp)\subset S$. Therefore, $\Supp(Sp)$ is contained in the intersection of all the maximum independent sets in $\mathcal{S}(Sp)$. Now, use \Cref{unb.max.clique.ind}(2).
\end{proof}

A \emph{threshold} graph is a graph that can be constructed from $K_1$ by repeatedly applying either of the following two operations: 1) addition of an isolated vertex; 2) addition of a universal vertex. Thus, each unlabeled threshold graph of order $n$ can be uniquely identified with a binary sequence $0b_2b_3\ldots b_n$, with $b_i\in\{0,1\}$, where $0=$``add an isolated vertex'' and $1=$``add a universal vertex''. Threshold graphs are characterized as those graphs that do not contain induced subgraphs isomorphic to $P_4$, $C_4$ or $2K_2$ (\cite{barrus.west.A4}, page 2). Since a graph is split if and only if it does not contain induced subgraphs isomorphic to $2K_2$, $C_4$ or $C_5$ (\cite{barrus.west.A4}, page 14), it is clear that threshold graphs are split. If $0b_2\ldots b_{n}$ is the binary sequence associated with a threshold graph $Sp$, it is easy to see that 
\[ (K,S)=(\{v\in V(Sp):b_v=1\},\{v\in V(Sp):b_v=0\}) \]
is an s-partition for $Sp$. Then, $N_{Sp}(v_1)=K$, which shows that $Sp$ is unbalanced. The following corollary is a refined version of a result that can be found in \cite{panelo5316214core}.

%%%%%%%%%%%%%%%%%%%%%%%%%%%%%%%%%%%%%%%%%%%%%%%%%
%
%
%
%%%%%%%%%%%%%%%%%%%%%%%%%%%%%%%%%%%%%%%%%%%%%%%%%

\begin{corollary}
	\label[corollary]{thresh.ind.supp}
	If $G$ is a threshold graph, then $\Supp(G)$ is contained in the intersection of all the maximum independent sets of $G$.
\end{corollary}

\begin{proof}
	It follows from the previous discussion and from \Cref{unbalanced.supp}.
\end{proof}

%%%%%%%%%%%%%%%%%%%%%%%%%%%%%%%%%%%%%%%%%%%%%%%%%
%
%
%
%%%%%%%%%%%%%%%%%%%%%%%%%%%%%%%%%%%%%%%%%%%%%%%%%

%%%%%%%%%%%%%%%%%%%%%%%%%%%%%%%%%%%%%%%%%%%%%%%%%
%
%
%
%%%%%%%%%%%%%%%%%%%%%%%%%%%%%%%%%%%%%%%%%%%%%%%%%
\begin{lemma}
	\label[lemma]{false.twins.x_u=x_v}
	Let $u,v$ be twin vertices of a graph $G$. If $x\in\nulo(G)$ and $uv\in E(G)$, then $x_u=x_v$.
\end{lemma}

\begin{proof}
	Since $x\in\nulo(G)$, we have
	\[ G_{v*}x=x_u+\sum_{i\in N(v)\setminus u}x_i = 0 = G_{u*}x=x_v+\sum_{i\in N(u)\setminus v}x_i. \]
	Since $N(v)\setminus u=N(u)\setminus v$, we have
	\[ \sum_{i\in N(v)\setminus u}x_i = \sum_{i\in N(u)\setminus v}x_i. \] 
	Therefore, $x_u=x_v$, as desired.
\end{proof}
%%%%%%%%%%%%%%%%%%%%%%%%%%%%%%%%%%%%%%%%%%%%%%%%%
%
%
%
%%%%%%%%%%%%%%%%%%%%%%%%%%%%%%%%%%%%%%%%%%%%%%%%%

\begin{theorem}
	Let $Sp$ be an unbalanced split graph. Then, we have the following:
	\begin{enumerate}
		%\item $(|K|,|S|)\in\{(\omega(Sp),\alpha(Sp)-1),(\omega(Sp)-1,\alpha(Sp))\}$;
		\item $W(Sp)=\{v\in V(Sp):\deg_{Sp}(v)=\omega(Sp)-1\}$;
		
		\item $W(Sp)$ is a clique or an independent set of mutually twin vertices;
		
		\item if $W(Sp)=\{w\}$, then $\Supp(Sp)\subset S^*(Sp)\cup w$;
		
		\item if $W(Sp)$ is a clique of size $\geq 2$, then $\Supp(Sp)\subset S^*(Sp)$; in particular, $W(Sp)\cap\Supp(Sp)=\varnothing$;
		
		\item if $W(Sp)$ is an independent set of size $\geq 2$, then 
		\[ W(Sp)\subset\Supp(Sp)\subset S^*(Sp)\cup W(Sp). \] 
	\end{enumerate}
\end{theorem}

\begin{proof}
	\begin{enumerate}[(1)]
		%\item We already know from previous discussions that this is equivalent to being unbalanced.
		
		\item It is the content of \Cref{ABC=WK*S*}.
		
		\item It is the content of \Cref{ABC=WK*S*}.
		
		\item It follows from \Cref{unbalanced.supp} and \Cref{cheng_collins_trenk}(1).
		
		\item From \Cref{cheng_collins_trenk}(2) we easily derive that the intersection of all the maximum independent sets of $Sp$ is $S^*$. Hence, $\Supp(Sp)\subset S^*$ by \Cref{unbalanced.supp}, and so $W\cap\Supp(Sp)=\varnothing$ by \eqref{vertex_tripartition}.
		
		Alternatively: if $x\in\nulo(Sp)$, then $x_u=x_v$ for all $u,v\in W$, by \Cref{false.twins.x_u=x_v} and \Cref{ABC=WK*S*}. Since $\Supp(Sp)$ is independent by \Cref{unbalanced.supp}, it must be $x_w=0$ for all $w\in W$, because $|W|\geq 2$ and $W$ is a clique. Consequently, $W\cap\Supp(Sp)=\varnothing$. 
		
		\item If $u,v$ are distinct swing vertices, observe that $A(Sp)_{*u}=A(Sp)_{*v}$. As a consequence, $e_u-e_v\in\nulo(Sp)$. Then, $W\subset\Supp(Sp)$. 
		
		On the other hand, we know from \Cref{cheng_collins_trenk}(3) that $S^*\cup W$ is the unique maximum independent set of $Sp$. Therefore, $\Supp(Sp)\subset S^*\cup W$, by \Cref{unbalanced.supp}. \qedhere
	\end{enumerate}
\end{proof}
%%%%%%%%%%%%%%%%%%%%%%%%%%%%%%%%%%%%%%%%%%%%%%%%%
%
%
%
%%%%%%%%%%%%%%%%%%%%%%%%%%%%%%%%%%%%%%%%%%%%%%%%%

%%%%%%%%%%%%%%%%%%%%%%%%%%%%%%%%%%%%%%%%%%%%%%%%%%%%%%%%%%%
%%%%%%%%%%%%%%%%%%%%%%%%%%%%%%%%%%%%%%%%%%%%%%%%%
%
%
%
%%%%%%%%%%%%%%%%%%%%%%%%%%%%%%%%%%%%%%%%%%%%%%%%%
%%%%%%%%%%%%%%%%%%%%%%%%%%%%%%%%%%%%%%%%%%%%%%%%%
%
%
%
%%%%%%%%%%%%%%%%%%%%%%%%%%%%%%%%%%%%%%%%%%%%%%%%%

\section{Split graphs with supported clique vertices}\label{sec:supp.clique}

In this section %we examine the structure of the nullspace when clique vertices are involved; this scenario gives rise to rich yet highly constrained linear dependencies. We
we prove that the subspace of kernel vectors with support in the clique (the clique-kernel) is at most one-dimensional. This result leads to a precise dichotomy: the overall nullity of the split graph is either equal to the nullity of its biadjacency matrix \(R\), or exactly one greater. This establishes a complete and computable classification of the nullity based purely on the dimension of this key subspace, tightly linking the combinatorial presence of supported clique vertices to the algebraic property of the graph's nullity.

The following result is crucial as it shows that when clique vertices appear in the support, they do so in essentially a one-dimensional way.%, which has important implications for the structure of the nullspace.
\begin{lemma}
	\label{lemma.clique.support}
	Let $Sp(K,S)$ be a split graph such that $|K|\geq 2$. If $\cliqueker{Sp}\neq \{0\}$, then $\dim(\cliqueker{Sp})=1$
\end{lemma}

\begin{proof}
	Since $\cliqueker{Sp}\neq \{0\}$, it is clear that $R\neq 0$. Let $\{b_1,\ldots,b_r\}$ be a basis for the subspace $\cliqueker{Sp}$. For each $i\in[r]$, we have 
	\[ (I-J)b_i=Rc_i, \] 
	for some $c_i\in\mathbb{R}^{|S|}$. Then, $(b_i,c_i)^t\in\nulo(Sp)$ by \Cref{kernel.split}, and hence $b_i\notin\nulo(J)$ by item (4) of \Cref{kernel.split} (since $\unovec^t b_i \neq 0$). 
	
	Now, assume that $\dim(\cliqueker{Sp})\geq 2$. Then we can write:
	\[ \frac{b_1}{\unovec^t b_1} - \frac{Rc_1}{\unovec^t b_1} = \unovec = \frac{b_2}{\unovec^t b_2} - \frac{Rc_2}{\unovec^t b_2}. \]
	
	Rearranging gives:
	\[ \frac{b_1}{\unovec^t b_1} - \frac{b_2}{\unovec^t b_2} = \frac{Rc_1}{\unovec^t b_1} - \frac{Rc_2}{\unovec^t b_2} \in \im(R). \]
	
	The left-hand side belongs to $\nulo(R^t)$ since $b_1, b_2 \in \cliqueker{Sp} \subset \nulo(R^t)$. Therefore,
	\[ \frac{b_1}{\unovec^t b_1} - \frac{b_2}{\unovec^t b_2} \in \nulo(R^t) \cap \im(R) = \{0\}, \]
	so
	\[ \frac{b_1}{\unovec^t b_1} = \frac{b_2}{\unovec^t b_2}, \]
	which implies $b_1$ and $b_2$ are linearly dependent, contradicting our assumption. Hence, $\dim(\cliqueker{Sp})=1$.
\end{proof}
%%%%%%%%%%%%%%%%%%%%%%%%%%%%%%%%%%%%%%%%%%%%%%%%%
%
%
%
%%%%%%%%%%%%%%%%%%%%%%%%%%%%%%%%%%%%%%%%%%%%%%%%%
Now we can show that the appearance of clique vertices in the support is an "all-or-nothing" phenomenon in a one-dimensional sense.%, which has significant implications for understanding the structure of kernel vectors in split graphs.
\begin{theorem}
	\label{clique.supp.characterization}
	Let $Sp(K,S)$ be a split graph such that $|K|\geq 2$. If $\cliqueker{Sp}\neq \{0\}$, then $\Supp(Sp)\cap K\neq\varnothing$ if and only if $\dim(\cliqueker{Sp})=1$.
\end{theorem}

\begin{proof}
	We obtain the desired result by combining item (3) of \Cref{kernel.split.basic.facts} with \Cref{lemma.clique.support}. 
\end{proof}
%%%%%%%%%%%%%%%%%%%%%%%%%%%%%%%%%%%%%%%%%%%%%%%%%
%
%
%
%%%%%%%%%%%%%%%%%%%%%%%%%%%%%%%%%%%%%%%%%%%%%%%%%

\begin{theorem}
	\label{clique.supp.nullity}
	If $Sp(K,S)$ is a split graph with $\Supp(Sp)\cap K\neq\varnothing$, then 
	\[ \nul(Sp)=\nul(R)+1. \]
\end{theorem}

\begin{proof}
	Let $\mathcal{X}=\{(0,x)^t:x\in\nulo(R)\}$. Since $\mathcal{X}\subset\nulo(Sp)$ by \Cref{kernel.split}(2), we can extend a basis $B_{\mathcal{X}}$ for $\mathcal{X}$ to a basis for $\nulo(Sp)$ to span those kernel vectors $(x_K,x_S)^t$ with $x_K\neq 0$.
	
	We know from \Cref{clique.supp.characterization} that $\cliqueker{Sp}=\langle z_0\rangle$, where $z_0=(I-J)^{-1}Ry$, for some $y$. Since $\nulo((I-J)^{-1}R)=\nulo(R)$ and $z_0\neq 0$, such a vector $y$ belongs to $y_0+\nulo(R)$, where $y_0$ is fixed and $y_0\notin\nulo(R)$.
	%where $z_0=(I-J)^{-1}Ry_0$ for some $y_0\in\mathbb{R}^{|S|}$ with $y_0\notin\nulo(R)$ 
	Then, by items (1) and (2) of \Cref{kernel.split.basic.facts}, every vector $s\in\nulo(Sp)\setminus\mathcal{X}$ has the form
	\[ s=c\binom{z_0}{y_0}+\binom{0}{x_0}, \]
	for some $c\in\mathbb{R}\setminus 0$ and some $x_0\in\nulo(R)$. 
	%Therefore, $B_{\mathcal{R}}\dot{\cup}\{(z_0,y_0)^t\}$ is a basis for $\ker(S)$.
%	By items (1) and (2) of \Cref{kernel.split.basic.facts}, for each vector $s\in\nulo(Sp)\setminus\mathcal{R}$, we have $s_K\in\cliqueker{Sp}$, so $s_K = cz_0$ for some $c\in\mathbb{R}\setminus\{0\}$. Moreover, there exists some $y\in\mathbb{R}^{|S|}$ such that $(z_0,y)^t\in\nulo(Sp)$. 
%	
%	Any such vector $s$ can be written as:
%	\[ s = c\binom{z_0}{y_0} + \binom{0}{x_0}, \]
%	for some $c\in\mathbb{R}\setminus\{0\}$ and $x_0\in\nulo(R)$, where we fix a particular representative $y_0$ such that $(z_0,y_0)^t\in\nulo(Sp)$. 
	Therefore, $B_{\mathcal{X}}\dot{\cup}\{(z_0,y_0)^t\}$ is a basis for $\nulo(Sp)$.
	%and since $|B_{\mathcal{X}}|=\nul(R)$, we conclude that $\nul(Sp)=\nul(R)+1$.
\end{proof}
%%%%%%%%%%%%%%%%%%%%%%%%%%%%%%%%%%%%%%%%%%%%%%%%%
%
%
%
%%%%%%%%%%%%%%%%%%%%%%%%%%%%%%%%%%%%%%%%%%%%%%%%%

\begin{corollary}
	If $Sp(K,S)$ is a split graph with $|K|\geq 2$, then $\dim(\cliqueker{Sp})$ $\leq 1$ and
	\[ \nul(Sp)=
	\begin{cases}
		\nul(R), & \text{if } \dim(\cliqueker{Sp})=0 \\
		\nul(R)+1, & \text{if } \dim(\cliqueker{Sp})=1.
	\end{cases}
	\]
\end{corollary}

\begin{proof}
	By \Cref{lemma.clique.support}, $\dim(\cliqueker{Sp})\leq 1$. If $\dim(\cliqueker{Sp})=0$, then by \Cref{kernel.split.basic.facts}(3), $\Supp(Sp)\subset S$, and hence $\nul(Sp)=\nul(R)$ by \Cref{kernel.split}(3). If $\dim(\cliqueker{Sp})=1$, then by \Cref{clique.supp.characterization}, $\Supp(Sp)\cap K\neq\varnothing$, and hence $\nul(Sp)=\nul(R)+1$ by \Cref{clique.supp.nullity}.
\end{proof}

%%%%%%%%%%%%%%%%%%%%%%%%%%%%%%%%%%%%%%%%%%%%%%%%%
%
%
%
%%%%%%%%%%%%%%%%%%%%%%%%%%%%%%%%%%%%%%%%%%%%%%%%%

%%%%%%%%%%%%%%%%%%%%%%%%%%%%%%%%%%%%%%%%%%%%%%%%%
%
%
%
%%%%%%%%%%%%%%%%%%%%%%%%%%%%%%%%%%%%%%%%%%%%%%%%%

\section{Split graphs with nullity 1}\label{sec:null.1}

Having determined the general structure of the nullspace, we now specialize in split graphs whose nullity is exactly one. This condition imposes a remarkable rigidity on the graph's structure, forcing a highly constrained relationship between the kernel, the adjugate of the adjacency matrix, and the graph's support. 

In this section, we leverage the classical properties of the adjugate for singular matrices to derive precise combinatorial descriptions of these graphs, culminating in formulas that link the coordinates of the unique kernel vector to the determinants of principal submatrices.

\begin{figure}[h]
	\centering
	\begin{tikzpicture}[scale=2, every node/.style={scale=1.3, circle, draw, thick, minimum size=6pt, inner sep=2pt}]
		% --- Nodos de la clique K6 en forma de hexágono ---
		\foreach \i in {0,...,5} {
			\node[fill=black] (k\i) at ({cos(60*\i)}, {sin(60*\i)}) {};
		}
		
		% --- Conexiones completas dentro de K (clique) ---
		\foreach \i in {0,...,5} {
			\foreach \j in {\i,...,5} {
				\ifnum\i<\j
				\draw[thick] (k\i) -- (k\j);
				\fi
			}
		}
		
		% --- Vértices independientes (I1,...,I4) ---
		\node[fill=white] (i1) at (-2, 0.8) {};
		\node[fill=white] (i2) at (-2, -0.2) {};
		\node[fill=white] (i3) at (2, 0.4) {};
		\node[fill=white] (i4) at (2, -0.6) {};
		
		% --- Aristas según R (filas = K, columnas = I) ---
		% Fila 1: 1 1 1 0
		\draw (k0) -- (i1);
		\draw (k0) -- (i2);
		\draw (k0) -- (i3);
		
		% Fila 2: 0 1 1 1
		\draw (k1) -- (i2);
		\draw (k1) -- (i3);
		\draw (k1) -- (i4);
		
		% Fila 3: 0 0 1 0
		\draw (k2) -- (i3);
		
		% Fila 4: 1 0 0 1
		\draw (k3) -- (i1);
		\draw (k3) -- (i4);
		
		% Fila 5: 0 1 0 0
		\draw (k4) -- (i2);
		
		% Fila 6: 0 1 0 0
		\draw (k5) -- (i2);
	\end{tikzpicture}
	\caption{Split graph with supported clique vertices and nullity one.}
	\label{fig:splitR}
\end{figure}
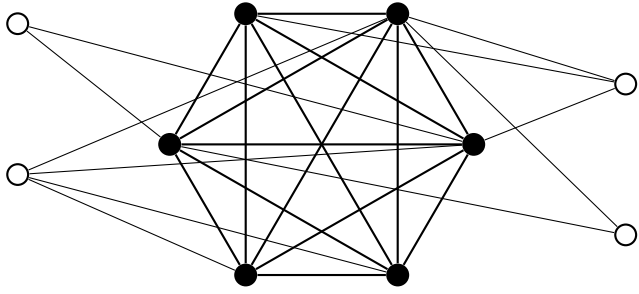

\begin{figure}[h]
	\centering
	\begin{tikzpicture}[scale=1.5, every node/.style={circle, draw, thick, minimum size=6pt, inner sep=2pt}]
		% --- Vértices de I (independientes, blancos) ---
		\node[fill=white, label=above:$1$] (i7)  at (0, 2) {7};
		\node[fill=white, label=above:$2$] (i8)  at (1.5, 2) {8};
		\node[fill=white, label=above:$1$] (i9)  at (3, 2) {9};
		\node[fill=white, label=above:$1$] (i10) at (4.5, 2) {10};
		
		% --- Vértices de K (clique, negros) ---
		\node[fill=black, text=white, label=below:$1$] (k1) at (0, 0) {1};
		\node[fill=black, text=white, label=below:$1$] (k2) at (1, 0) {2};
		\node[fill=black, text=white, label=below:$-2$] (k3) at (2, 0) {3};
		\node[fill=black, text=white, label=below:$-1$] (k4) at (3, 0) {4};
		\node[fill=black, text=white, label=below:$-1$] (k5) at (4, 0) {5};
		\node[fill=black, text=white, label=below:$-1$] (k6) at (5, 0) {6};
		
		% --- Aristas según la matriz R ---
		% Fila 1: 1 1 1 0
		\draw (k1) -- (i7);
		\draw (k1) -- (i8);
		\draw (k1) -- (i9);
		
		% Fila 2: 0 1 1 1
		\draw (k2) -- (i8);
		\draw (k2) -- (i9);
		\draw (k2) -- (i10);
		
		% Fila 3: 0 0 1 0
		\draw (k3) -- (i9);
		
		% Fila 4: 1 0 0 1
		\draw (k4) -- (i7);
		\draw (k4) -- (i10);
		
		% Fila 5: 0 1 0 0
		\draw (k5) -- (i8);
		
		% Fila 6: 0 1 0 0
		\draw (k6) -- (i8);
	\end{tikzpicture}
	\caption{Bipartite representation of the split graph $Sp(K,S)$ of \Cref{fig:splitR}. The upper row corresponds to $S$ and the lower row to $K$ (black vertices). Since $\nulo(Sp)=\langle z\rangle$, the coordinates of $z$ are shown.}
	\label{fig:splitR_bipartito}
\end{figure}
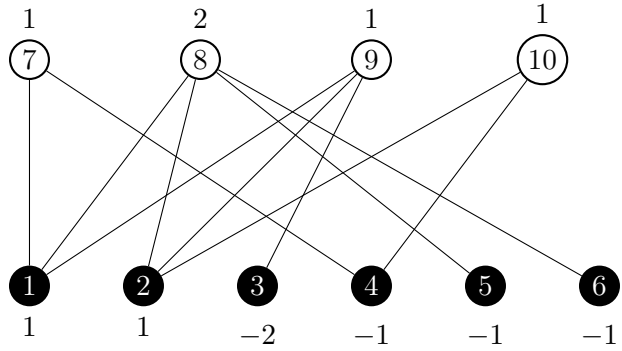

\begin{figure}[h]
	\label{fig:splitR_bipartito2}
	\centering
	\begin{tikzpicture}[scale=1.4, every node/.style={circle, draw, thick, minimum size=6pt, inner sep=2pt}]
		% --- Vértices de I (independientes, blancos) ---
		\node[fill=white, label=above:$-1$] (i6) at (0, 2) {6};
		\node[fill=white, label=above:$-1$] (i7) at (1.5, 2) {7};
		\node[fill=white, label=above:$1$] (i8) at (3, 2) {8};
		\node[fill=white, label=above:$-1$] (i9) at (4.5, 2) {9};
		
		% --- Vértices de K (clique, negros) ---
		\node[fill=black, text=white, label=below:$1$]  (k1) at (0, 0) {1};
		\node[fill=black, text=white, label=below:$0$] (k2) at (1.2, 0) {2};
		\node[fill=black, text=white, label=below:$1$] (k3) at (2.4, 0) {3};
		\node[fill=black, text=white, label=below:$1$]  (k4) at (3.6, 0) {4};
		\node[fill=black, text=white, label=below:$-1$]  (k5) at (4.8, 0) {5};
		
		% --- Aristas según la matriz R ---
		% Fila 1: 1 0 0 0
		\draw (k1) -- (i6);
		
		% Fila 2: 1 1 1 1
		\draw (k2) -- (i6);
		\draw (k2) -- (i7);
		\draw (k2) -- (i8);
		\draw (k2) -- (i9);
		
		% Fila 3: 0 1 0 0
		\draw (k3) -- (i7);
		
		% Fila 4: 0 0 0 1
		\draw (k4) -- (i9);
		
		% Fila 5: 1 1 0 1
		\draw (k5) -- (i6);
		\draw (k5) -- (i7);
		\draw (k5) -- (i9);
	\end{tikzpicture}
	\caption{Bipartite representation of a split graph  (clique vertices in black) with nullity one. The numbers next to each vertex indicate the coordinates of ``the'' spanning kernel vector.}
\end{figure}
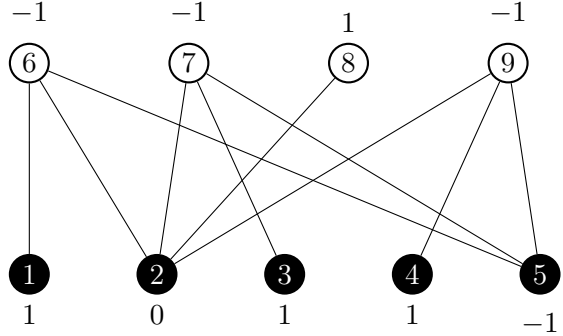
%%%%%%%%%%%%%%%%%%%%%%%%%%%%%%%%%%%%%%%%%%%%%%%%%
%
%
%
%%%%%%%%%%%%%%%%%%%%%%%%%%%%%%%%%%%%%%%%%%%%%%%%%

\begin{lemma}[\cite{horn2012matrix}]
	\label{lemma.adj.nonzero}
	If $A$ is a square matrix, then $\adj(A)\!\neq \!0$ if and only if $\nul(A)\!\leq\! 1$. 
\end{lemma}
While this is a standard result (see \cite{horn2012matrix}), we include a concise proof here as this lemma plays a crucial role in our characterization of split graphs with nullity one, and we wish to keep our presentation self-contained.
\begin{proof}
	Let $n$ be the order of $A$. The fact that $\adj(A)$ has at least a nonzero entry means that $A$ contains a nonzero minor of order $n-1$. This is equivalent to $\rank(A)\geq n-1$, because $\rank(A)$ is, by definition, the maximum size of an invertible submatrix of $A$.
%	Let $n$ be the order of $A$. Recall that $\rank(\adj(A))$ is $n$ if $\rank(A)=n$, $1$ if $\rank(A)=n-1$, and $0$ if $\rank(A)\leq n-2$. Thus $\adj(A)\neq 0$ if and only if $\rank(A)\geq n-1$, which is equivalent to $\nul(A)\leq 1$.
\end{proof}

%%%%%%%%%%%%%%%%%%%%%%%%%%%%%%%%%%%%%%%%%%%%%%%%%
%
%
%
%%%%%%%%%%%%%%%%%%%%%%%%%%%%%%%%%%%%%%%%%%%%%%%%%

\begin{lemma}
	\label{null=1.adj(S).columns}
	Let $A$ be a square matrix. If $\nul(A)=1$, then $\rank(\adj(A))=1$ and every column of $\adj(A)$ is a vector of $\nulo(A)$.	   
\end{lemma}
This also is a standard result in linear algebra; we include the proof for completeness.
\begin{proof}
	Let $n$ be the order of $A$ and let $\adj(A)=[a_1 \mid \ldots \mid a_n]$, where $a_i$ denotes the $i$-th column. Since $\det(A)=0$, the fundamental identity $\det(A)I=A\adj(A)$ implies that
	\[
	0=A\adj(A)=[Aa_1 \mid \ldots \mid Aa_n].
	\]
	So, $a_i\in\nulo(A)$ for all $i\in[n]$, from which we deduce that $\rank(\adj(A))\leq 1$, because $\nul(A)=1$. On the other hand, we know by \Cref{lemma.adj.nonzero} that $\rank(\adj(A))\neq 0$. Therefore, $\rank(\adj(A))=1$.
	%Since $\rank(A)=n-1$, there exists an invertible $(n-1)\times(n-1)$ submatrix $A'$ of $A$. The corresponding cofactor (which appears as an entry in $\adj(A)$) is $\pm\det(A')\neq 0$, showing that $\adj(A)\neq 0$. 
\end{proof}

If $A$ is a symmetric matrix, recall that $\adj(A)$ is symmetric as well.

\begin{theorem}
	\label{G.nul=1.adj.G-v.supp.x_v^2}
	Let $G$ be a graph with $\nul(G)=1$ and $\Supp(G)=[s]$, for some $s\in\mathbb{N}$. Then, we have the following:
	\begin{enumerate}
		\item $G-v$ is singular if and only if $\adj(G)_{vv}= 0$;
		\item $\adj(G)_{uv}\neq 0$ if and only if $u,v\in\Supp(G)$;
		\item $v\in\Supp(G)$ if and only if $G-v$ is nonsingular;
		\item If $(x,0)^t\in\nulo(G)$, with $x\in\mathbb{R}^s\setminus 0$, then 
		\[ x_v^2=\det(G-v)\det(G-u), \]
		for each $v\in\Supp(G)$ and some fixed $u\in\Supp(G)$.
	\end{enumerate}
\end{theorem}

\begin{proof}
	\begin{enumerate}[(1)]
		\item By removing row $v$ and column $v$ from $A(G)$ we obtain $A(G-v)$. Since $\adj(G)_{vv}=(-1)^{v+v}\det(G-v)=\det(G-v)$, we get the desired equivalence.
		
		\item Let $\nulo(G)=\langle z\rangle$ and $s=|\Supp(G)|$. First, we label $V(G)$ using the ordering $(\Supp(G),V(G)\setminus\Supp(G))$. By this labeling, $z=(x,0)^t$, where $x\in\mathbb{R}^{s}$ and $x_v\neq 0$ for all $v\in[s]$. Then, by \Cref{null=1.adj(S).columns}, we can write
		\[  
		\adj(G)=
		\begin{pmatrix}
			Z & 0 \\
			0 & 0
		\end{pmatrix},
		\]
		where 
		\[ Z=[x \mid c_2x \mid \ldots \mid c_sx]=x(1,c_2,\ldots,c_s), \]
		for some nonzero constants $c_2,\ldots,c_s$. So, we see that $\adj(G)_{uv}\neq 0$ if and only if $u,v\in\Supp(G)$.
		
		\item It follows from (1) and (2).
		
		\item Since $A(G)$ is symmetric, $\adj(G)$ is symmetric, so $Z$ must be symmetric. From the structure of $Z=x(1,c_2,\ldots,c_s)$, symmetry implies that $c_i = \frac{x_i}{x_1}$ for all $i\in[s]$. Therefore, $Z = x_1^{-1}xx^t$. 
		
		Fix $u\in\Supp(G)$ as the first vertex in our ordering, so $x_1=\adj(G)_{uu}=\det(G-u)$. Then for each $v\in\Supp(G)$:
		\[ \det(G-v)=\adj(G)_{vv}=Z_{vv}=x_1^{-1}x_v^2=\det(G-v_1)^{-1}x_v^2. \]
		Rearranging gives $x_v^2=\det(G-v)\det(G-u)$. \qedhere
	\end{enumerate}
\end{proof}

%%%%%%%%%%%%%%%%%%%%%%%%%%%%%%%%%%%%%%%%%%%%%%%%%
%
%
%
%%%%%%%%%%%%%%%%%%%%%%%%%%%%%%%%%%%%%%%%%%%%%%%%%
The former result allow us to prove that for split graph with nullity one all the ``interesting'' kernel structure comes from the clique-support phenomenon, not from $R$ itself.
\begin{corollary}
	\label{null(S)=1.iff.ker(R)=0}
	Let $Sp(K,S)$ be a split graph such that $\Supp(Sp)\cap K\neq\varnothing$, with $|K|\geq 2$. Then, $\nul(Sp)=1$ if and only if $\nulo(R)=\{0\}$.
\end{corollary}

\begin{proof}
	By \Cref{clique.supp.nullity}, since $\Supp(Sp)\cap K\neq\varnothing$, we have $\nul(Sp)=\nul(R)+1$. Therefore, $\nul(Sp)=1$ if and only if $\nul(R)=0$.
\end{proof}
%%%%%%%%%%%%%%%%%%%%%%%%%%%%%%%%%%%%%%%%%%%%%%%%%
%
%
%
%%%%%%%%%%%%%%%%%%%%%%%%%%%%%%%%%%%%%%%%%%%%%%%%%

\begin{lemma}
	\label{A.null=1.full.supported}
	Let $A$ be an $m\times n$ matrix such that $n\geq 2$, $\nul(A)=1$, and $A_{*j}\neq 0$ for all $j\in[n]$. If $z\in\nulo(A)\setminus 0$ and $z$ has no zero entries, then every set of $n-1$ columns of $A$ is a basis for $\im(A)$.
\end{lemma}

\begin{proof}
	Since $z\in\nulo(A)$ has no zero entries, for each $k\in[n]$ we have $A_{*k} = -\frac{1}{z_k}\sum_{j\neq k} z_jA_{*j}$, so each $B_i=\{A_{*j}:j\neq i\}$ spans $\im(A)$. Since $\rank(A)=n-1=|B_i|$, each $B_i$ is a basis for $\im(A)$.
\end{proof}
%%%%%%%%%%%%%%%%%%%%%%%%%%%%%%%%%%%%%%%%%%%%%%%%%
%
%
%
%%%%%%%%%%%%%%%%%%%%%%%%%%%%%%%%%%%%%%%%%%%%%%%%%

\begin{lemma}
	\label{basis.Im(A).null=1}
	Let $A$ be an $m\times n$ matrix with $|\Supp(A)|\geq 2$, $\nul(A)=1$, and $A_{*j}\neq 0$ for all $j\in[n]$. If $i\in\Supp(A)$, define $B_i=\{A_{*j}:j\in\Supp(A),j\neq i\}$. Then, 
	\[ B = B_i \cup \{ A_{*j} : j\notin\Supp(A) \} \] 
	is a basis for $\im(A)$, for each $i\in\Supp(A)$.
\end{lemma}

\begin{proof}
	From \Cref{A.null=1.full.supported}, we know that $B_i$ spans the subspace $\langle \{A_{*j}:j\in\Supp(A)\} \rangle$. With this in mind, it is clear that $B$ spans $\im(A)$. Since $|B|=n-1=\rank(A)$, we can finally conclude that $B$ is a basis for $\im(A)$.
%	Let $s = |\Supp(A)| \geq 2$. Consider the restriction of $A$ to the columns in $\Supp(A)$, and let $z_{\Supp}$ be the restriction of a nonzero kernel vector $z$ to these coordinates. Since $z_j \neq 0$ for all $j \in \Supp(A)$, we can apply \Cref{A.null=1.full.supported} to conclude that $B_i$ spans $\langle \{A_{*j} : j \in \Supp(A)\} \rangle$.
%	
%	The set $B$ contains:
%	- $B_i$: $s-1$ vectors spanning the columns from $\Supp(A)$
%	- $\{A_{*j} : j \notin \Supp(A)\}$: $n-s$ vectors
%	
%	Thus $|B| = (s-1) + (n-s) = n-1 = \rank(A)$. Since $B$ spans $\im(A)$ and has the correct cardinality, it is a basis for $\im(A)$.
\end{proof}
%%%%%%%%%%%%%%%%%%%%%%%%%%%%%%%%%%%%%%%%%%%%%%%%%
%
%
%
%%%%%%%%%%%%%%%%%%%%%%%%%%%%%%%%%%%%%%%%%%%%%%%%%

\begin{theorem}
	\label{basis.Im(G).null=1}
	Let $G$ be a graph without isolated vertices. If $\nul(G)=1$ and $s\in\Supp(G)$, then 
	\[ B_s = \{A(G)_{*j} : j \neq s\} \]
	is a basis for $\im(G)$.
\end{theorem}
This theorem provides a very concrete and useful result: for any graph with nullity 1 and no isolated vertices, you can obtain a basis for the column space by simply removing any one column corresponding to a supported vertex. %This is computationally useful and reveals the highly structured nature of the adjacency matrix in this case.
\begin{proof}
	Since $G$ has no isolated vertices, $A(G)_{*j} \neq 0$ for all $j \in V(G)$. If $\Supp(G)=\{u\}$ for some $u$, then $\nulo(G)=\langle e_u\rangle$. But, $0=A(G)e_u=A(G)_{*u}$, contradicting that all the columns of $A(G)$ are nonzero. Therefore, $|\Supp(G)| \geq 2$, and so we can apply \Cref{basis.Im(A).null=1}.
\end{proof}

%If we can find a supported vertex in a graph $G$ with nullity 1, then by \Cref{basis.Im(A).null=1} we automatically obtain a basis for $\im(G)$. This means that the core-nilpotent decomposition is very easy for $G$.

\begin{lemma}
	\label{lemma.ker(P^t)=ker(R^t)}
	Let $Sp(K,S)$ be a split graph such that $\cliqueker{Sp}\neq \{0\}$,	and let $P$ be the submatrix of $R$ of size $|K|\times|\rank(R)|$ whose columns form a basis for $\im(R)$. Let $S'$ be the set of vertices associated with the columns of $P$. If $Sp'(K,S')$ is the split graph whose biadjacency matrix is $P$, then 
	\[ \cliqueker{Sp}=\cliqueker{Sp'} \]
	and $\nul(Sp')=1$.  
\end{lemma}

\begin{proof}
	Since $P$ is obtained by selecting a basis for $\im(R)$, we have $\im(P) = \im(R)$ (hence, $\im((I-J)^{-1}P)=\im((I-J)^{-1}R)$) and
	$\nulo(P^t) = \nulo(R^t)$. Therefore, $\cliqueker{Sp} =\cliqueker{Sp'}$. Since $\cliqueker{Sp}\neq \{0\}$, by \Cref{kernel.split.basic.facts}(3), we have $\Supp(Sp')$ $\cap K\neq\varnothing$. Since $\nulo(P)=\{0\}$, applying \Cref{null(S)=1.iff.ker(R)=0} to $Sp'$ gives $\nul(Sp')=1$.
\end{proof}
%%%%%%%%%%%%%%%%%%%%%%%%%%%%%%%%%%%%%%%%%%%%%%%%%
%
%
%
%%%%%%%%%%%%%%%%%%%%%%%%%%%%%%%%%%%%%%%%%%%%%%%%%

\begin{theorem}
	\label{basis.im.ker.split.supp.clique}
	Let $Sp(K,S)$ be a split graph such that $\cliqueker{Sp}=\langle z\rangle$, for some $z\in\mathbb{R}^{|K|}\setminus 0$. Assume that the first $\rank(R)$ columns of $R$ are linearly independent and let $P$ be the submatrix of $R$ obtained by deleting the last $|S|-\rank(R)$ columns of $R$. Let \(S'\) be the set of vertices associated with the columns of \(P\). If $Sp'(K,S')$ is the induced split subgraph of $Sp$ whose biadjacency matrix is $P$, then we have the following:
	\begin{enumerate}
		\item $\Supp(Sp')\cap K=\Supp(Sp)\cap K$, \ $\Supp(Sp')\cap S'\subset\Supp(Sp)\cap S'$.
		\item For each $s\in\Supp(Sp')$, then
		\[ B_s=\{A(Sp)_{*j}:j\in V(Sp')\setminus s\} \] 
		is a basis for $\im(Sp)$.
		\item $\rank(Sp)=|V(Sp')|-1=\rank(R)+|K|-1$.
		\item If $y_0$ is the unique vector such that $(I-J)z=Py_0$, and $y_v$ is the unique vector such that $Py_v=R_{*v}$ for each $v\in W=V(Sp)\setminus V(Sp')$, then 
		\[ B=\{(z,y_0,0)^t\}\cup\{ (0,-y_v,e_v)^t:v\in W \} \]
		is a basis for $\nulo(Sp)$.   
	\end{enumerate}
\end{theorem}

\begin{proof}
	\begin{enumerate}[(1)]
		\item The equality $\Supp(Sp')\cap K=\Supp(Sp)\cap K$ can easily be derived from \Cref{lemma.ker(P^t)=ker(R^t)}. To obtain that $\Supp(Sp')\cap S'\subset\Supp(Sp)\cap S'$, recall that, by item (2) of \Cref{kernel.split.basic.facts} and item (1) of \Cref{kernel.split}, there exists a (unique) vector $y_0$ such that $(I-J)^{-1}Py_0=z$. Then, 
		\[ \Supp(Sp')\cap S'=\{u\in S':(y_0)_u\neq 0\}. \]
		On the other hand, we have $(I-J)^{-1}R(y_0,0)^t=z$ as well, which means that $(z,y_0,0)^t\in\nulo(Sp)$. Hence, 
		\[ \{u\in S':(y_0)_u\neq 0\}\subset\Supp(Sp)\cap S'. \]
		
		\item Let $M$ be the $|V(Sp)|\times(|V(Sp')|-1)$ matrix obtained from $A(Sp)$ by deleting its last $|S|-\rank(R)$ columns and the column $s$. Since $\rank(M)$ is the size of its largest invertible submatrix, we have that $\rank(M)=|V(Sp')|-1$ because $M$ contains $A(Sp'-s)$, which is nonsingular by \Cref{lemma.ker(P^t)=ker(R^t)} and \Cref{G.nul=1.adj.G-v.supp.x_v^2}. Hence, $B_s$ is a linearly independent set. 
		
		Obviously, $\langle B_s\rangle\subset\im(Sp)$. Since $s\in\Supp(Sp)$ by (1), there exists $x\in\nulo(Sp)$ such that $x_s\neq 0$ and $x_v=0$ for all $v\in V(Sp)\setminus V(Sp')$ (recall the proof of (1)). Therefore, we can solve for $A(Sp)_{*s}$ from the equality
		\[ 0=\sum_{j\in V(Sp)}x_jA(Sp)_{*j}=x_sA(Sp)_{*s}+\sum_{j\in V(Sp')\setminus s}x_jA(Sp)_{*j}, \]
		and show that $A(Sp)_{*s}\in\langle B_s\rangle$. For each $v\in V(Sp)\setminus V(Sp')$, it is clear by definition of $P$ that $A(Sp)_{*v}$ can be written as a nontrivial linear combination of elements in $\{A(Sp)_{*j}:j\in S'\}$ (if $s\in S'$, recall that $A(Sp)_{*s}\in\langle B_s\rangle$). Hence, $\langle B_s\rangle=\im(Sp)$.
		
		\item By (2), $\rank(Sp)=|B_s|=|V(Sp')|-1$. Moreover, $|V(Sp')|=|S'|+|K|$ and $|S'|=\rank(R)$, by the definition of $Sp'$. 
		
		\item Since 
		\[ 0=P(-y_v)+1R_{*v}+\sum_{j\in W\setminus v}0R_{*j}=R(-y_v,e_v)^t, \]
		we see that $(-y_v,e_v)^t\in\nulo(R)$. Note that $\{(-y_v,e_v):v\in W\}$ is clearly a basis for $\nulo(R)$ (recall that $|W|=\nul(R)$). Consequently, we get the desired result by repeating the same argument used in the proof of \Cref{clique.supp.nullity}. \qedhere
	\end{enumerate}
\end{proof}
%%%%%%%%%%%%%%%%%%%%%%%%%%%%%%%%%%%%%%%%%%%%%%%%%
%
%
%
%%%%%%%%%%%%%%%%%%%%%%%%%%%%%%%%%%%%%%%%%%%%%%%%%

%Up to sign, the integer vector \(z\) constructed in \Cref{basis.im.ker.split.supp.clique} is an algebraic invariant of split graphs with supported clique vertices.

Clearly, the vector $z$ constructed in \Cref{basis.im.ker.split.supp.clique} can always be chosen with integer coordinates such that $\gcd\{z_i:z_i\neq 0\}= 1$. In such a case, observe that $z$ becomes, up to multiplication by $-1$, an interesting algebraic invariant of split graphs with supported clique vertices.

%%%%%%%%%%%%%%%%%%%%%%%%%%%%%%%%%%%%%%%%%%%%%%%%%%%%%%%%%%%
%%%%%%%%%%%%%%%%%%%%%%%%%%%%%%%%%%%%%%%%%%%%%%%%%
%
%
%
%%%%%%%%%%%%%%%%%%%%%%%%%%%%%%%%%%%%%%%%%%%%%%%%%
%%%%%%%%%%%%%%%%%%%%%%%%%%%%%%%%%%%%%%%%%%%%%%%%%
%
%
%
%%%%%%%%%%%%%%%%%%%%%%%%%%%%%%%%%%%%%%%%%%%%%%%%%

\section{Tyshkevich composition}\label{sec:tyshk.null}

Understanding how the nullspace behaves under graph operations is a natural and important question. In this section, we analyze how the nullspace of a split graph interacts with the Tyshkevich composition, a well-known product that builds a larger graph by fully joining the clique of a split graph to an arbitrary second graph. %We establish fundamental results showing how the singularity of the composition is controlled by the properties of its components, providing a powerful tool for constructing graphs with prescribed nullity and for understanding the propagation of kernel vectors through this structured gluing.

If $Sp(K,S)$ is a split graph and $G$ is a graph disjoint from $Sp$, the \emph{Tyshkevich composition} $Sp\circ G$ of $Sp$ and $G$ is defined as the graph whose vertex set is 
\[ V(Sp\circ G)=V(Sp)\cup V(G), \]
and whose edge set is 
\[ E(Sp\circ G)=E(Sp)\cup E(G)\cup\{xy:x\in K,y\in V(G)\}. \]
This operation was introduced by R. Tyshkevich in \cite{tyshkevich2000decomposition}. In Figure \ref{ejemplo.composicion} we can see an example of this composition.

This operation is relevant for studying the nullspace of split graphs, as it allows us to build more complex split graphs from simpler components while potentially controlling spectral properties.

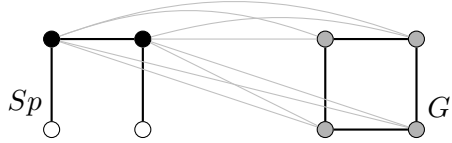
\begin{figure}[h]
	\label{ejemplo.composicion}
	\centering
	\begin{tikzpicture}[scale=1.2, every node/.style={circle, draw, minimum size=6pt, inner sep=2pt}]
		
		% Vértices de S en forma de U rotada 90°
		\node[fill=white, label=above left:$Sp$] (v1) at (0,0) {};    % hoja
		\node[fill=black] (v2) at (0,1) {};    % clique
		\node[fill=black] (v3) at (1,1) {};    % clique
		\node[fill=white] (v4) at (1,0) {};    % hoja
		
		% Aristas de S (P4)
		\draw[thick] (v1) -- (v2) -- (v3) -- (v4);
		
		% Vértices de G (C4), a la derecha de S
		\node[fill=gray!60] (g1) at (3,0) {};
		\node[fill=gray!60, label=above right:$G$] (g2) at (4,0) {};
		\node[fill=gray!60] (g3) at (4,1) {};
		\node[fill=gray!60] (g4) at (3,1) {};
		
		% Aristas de G (C4)
		\draw[thick] (g1) -- (g2) -- (g3) -- (g4) -- (g1);
		
		% Aristas de la composición: K × V(G)
		\foreach \k in {v2,v3} {
			\foreach \g in {g1,g2} {
				\draw [color=gray!50](\k) -- (\g);
			}
		}
		\draw[color=gray!50] (v3)--(g4);
		\draw[color=gray!50, bend left=15] (v3) to (g3);
		\draw[color=gray!50, bend left=20] (v2) to (g3);
		\draw[color=gray!50, bend left=20] (v2) to (g4);

		% Etiquetas opcionales
		%		\node[draw=none, anchor=east] at (v1.west) {$v_1$};
		%		\node[draw=none, anchor=east] at (v2.west) {$v_2$};
		%		\node[draw=none, anchor=south] at (v3.north) {$v_3$};
		%		\node[draw=none, anchor=north] at (v4.south) {$v_4$};
		%		
		%		\node[draw=none, anchor=north] at (g1.south) {$g_1$};
		%		\node[draw=none, anchor=north] at (g2.south) {$g_2$};
		%		\node[draw=none, anchor=south] at (g3.north) {$g_3$};
		%		\node[draw=none, anchor=south] at (g4.north) {$g_4$};
		
	\end{tikzpicture}
	\caption{Tyshkevich composition of $Sp\approx P_4$ and $G\approx C_4$.}
\end{figure}
%%%%%%%%%%%%%%%%%%%%%%%%%%%%%%%%%%%%%%%%%%%%%%%%%
%
%
%
%%%%%%%%%%%%%%%%%%%%%%%%%%%%%%%%%%%%%%%%%%%%%%%%%
 The Tyshkevich composition preserves (and potentially amplifies) the singularity coming from the biadjacency matrix $R$ of the split graph. 
\begin{theorem}
	\label{thm:tyshkevich.nullspace.containment}
	Let $Sp(K,S)$ be a split graph, and let $G$ be a graph. Then, 
	\[ \nulo(Sp\circ G)\supseteq\{ (0,z,0)^t:z\in\nulo(R) \}. \]
	In particular, if $\nulo(R)\neq \{0\}$, then $Sp\circ G$ is singular.
\end{theorem}

\begin{proof}
	Let $H=Sp\circ G$. By using the ordering $(K,S,V(G))$ for $V(H)$, we see that   
	\[
	A(H) = 
	\begin{pmatrix}
		J-I & R & J \\
		R^t & 0 & 0 \\
		J^t & 0 & A(G)
	\end{pmatrix}.
	\]
	Now, observe that $A(H)(0,x_S,0)^t=(Rx_S,0,0)^t$. Therefore, if $x_S\in\nulo(R)$, then $(0,x_S,0)^t\in\nulo(H)$.
\end{proof}
%%%%%%%%%%%%%%%%%%%%%%%%%%%%%%%%%%%%%%%%%%%%%%%%%
%
%
%
%%%%%%%%%%%%%%%%%%%%%%%%%%%%%%%%%%%%%%%%%%%%%%%%%

\begin{theorem}
	\label{thm:tyshkevich.nonsingular.equivalence}
	Let $Sp(K,S)$ and $Sp'(K',S')$ be split graphs with $|K|=|S|$ and $|K'|=|S'|$. Then, $Sp\circ Sp'$ is nonsingular if and only if $Sp$ and $Sp'$ are both nonsingular.
\end{theorem}

\begin{proof}
	Since $(K\cup K', S\cup S')$ is an s-partition for $H=Sp\circ Sp'$, we can label $V(H)$ following the ordering $(K,K',S,S')$, and write
	\[ A(H)=
	\begin{pmatrix}
		J-I & P \\
		P^t & 0
	\end{pmatrix},
	\quad P=
	\begin{pmatrix}
		R & J\\
		0 & R'
	\end{pmatrix}.
	\]
	Since $|K\cup K'|=|S\cup S'|$, we know by item (2) of \Cref{|K|<|I|.singular} that $A(H)$ is nonsingular if and only if 
	\[ \det(P) = \det(R)\det(R')\neq 0, \]
	because $P$ is block-triangular. Hence, $H$ is nonsingular if and only if $R$ and $R'$ are both nonsingular. Now, apply once again item (2) of \Cref{|K|<|I|.singular} to conclude that this occurs if and only if $Sp$ and $Sp'$ are both nonsingular.
\end{proof}

%%%%%%%%%%%%%%%%%%%%%%%%%%%%%%%%%%%%%%%%%%%%%%%%%%%%%%%%%%%

\section{Determinant of split graphs}\label{sec:det.split}
 
Thanks to the block structure of their adjacency matrices, split graphs admit explicit combinatorial determinant formulas that can be obtained with the Schur complement and Cauchy's formula.
 
 In this section, we establish a precise formula for \(\det(Sp)\) (i.e., $\det(A(Sp))$) in terms of the biadjacency matrix \(R\), providing a direct computational method to determine the singularity of a split graph and revealing how the determinant is influenced by the interplay between the clique and the independent set.
\begin{theorem}[Schur complement]
	\label{schur.complement}
	Let
	\[ M=
	\begin{pmatrix}
		A & B \\
		C & D
	\end{pmatrix}
	\]
	be a block matrix, where $A$ and $D$ are square matrices (of possibly different sizes). If $A$ is invertible, then
	\[ \det(M) = \det(A)\det(D-CA^{-1}B). \]
\end{theorem}

\begin{theorem}[Cauchy's formula, \cite{horn2012matrix}]
	\label{matrix.det.lemma}
	If $M$ is a square matrix and $u,v$ are vectors, then 
	\[ \det(M+uv^t)=\det(M)+v^t\adj(M)u. \]
\end{theorem}

\begin{theorem}
	\label{det.split}
	Let $Sp(K,S)$ be a split graph with $|K|=k\geq 2$. If 
	\[ A(Sp)= 
	\begin{pmatrix}
		J-I & R \\
		R^t & 0
	\end{pmatrix},
	\]
	and $r=R^t\unovec$, then
	\[ \det(Sp) = (-1)^{k-1}\left( (k-1)\det(R^tR)-r^t\adj(R^tR)r \right). \]
%	Furthermore, if $\nulo(R)=\{0\}$, we have
%	\[ \det\left(R^tR -\frac{rr^t}{k-1}\right) = \left( 1-\frac{1}{k-1}r^t(R^tR)^{-1}r \right)\det(R^tR). \] 
\end{theorem}

\begin{proof}
	We apply the Schur complement (\Cref{schur.complement}) with 
	
	$A=J-I$, $B=R$, $C=R^t$, and $D=0$:
	\[ \det(Sp) = \det(J-I)\det(0 - R^t(J-I)^{-1}R) =\] 
	\[ (-1)^{k-1}(k-1)\det(-R^t(J-I)^{-1}R). \]
	Using the well known Sherman--Morrison formula, $(I-J)^{-1}=I-(k-1)^{-1}J$, we have:
	\[ R^t(J-I)^{-1}R = -R^t(I-J)^{-1}R = -R^t\left(I - \frac{1}{k-1}J\right)R =\] 
	\[ -R^tR + \frac{1}{k-1}R^tJ R. \]
	Note that $R^tJ R = (R^t\unovec)(\unovec^t R) = rr^t$. Therefore:
	\[ \det(Sp) = (-1)^{k-1}(k-1)\det\left(R^tR - \frac{rr^t}{k-1}\right). \]
	Finally, we apply \Cref{matrix.det.lemma} with $M=R^tR, u=\frac{r}{1-k}$ and $v=r$:
%	For the second part, if $\nulo(R)=\{0\}$, then $R^tR$ is invertible, and we apply the matrix determinant lemma (\Cref{matrix.det.lemma}) with $M=R^tR, u=\frac{r}{1-k}$ and $v=r$.
	\[ \det\left(R^tR - \frac{rr^t}{k-1}\right) = \det(R^tR) - \frac{r^t\adj(R^tR)r}{k-1}.\qedhere \]
\end{proof}

Notice that the vector $r=R^t\unovec$ in \Cref{det.split} is the degree sequence of the independent part of $Sp(K,S)$, i.e., $r_v=\deg_{Sp}(v)$, for each $v\in S$.

\begin{corollary}
	\label{cor:split.singularity.condition}
	Let $Sp(K,S)$ be a split graph such that $k=|K|\geq 2$. If $r=R^t\unovec$, then $Sp$ is singular if and only if
	\[ (k-1)\det(R^tR)=r^t\adj(R^tR)r. \]
	In particular:
	\begin{enumerate}
		\item if $\nulo(R)=\{0\}$, then $\det(Sp)=0$ if and only if $r^t(R^tR)^{-1}r=k-1$;
		
		\item if $\nulo(R)=\langle x\rangle$ for some $x\neq 0$, then $\det(Sp)=0$ if and only if $r^tx=0$; 
		
		\item if $\nul(R)\geq 2$, then $\det(Sp)=0$.
	\end{enumerate}
\end{corollary}
\begin{proof}
	The singularity characterization immediately follows from \Cref{det.split}. 
	\begin{enumerate}[(1).]
		\item Since $R^tR$ is invertible, we have $\adj(R^tR)=(R^tR)^{-1}\det(R^tR)$.
		
		\item Since $\nulo(R^tR)=\nulo(R)$, we have $\det(R^tR)=0$ and we can write $\adj(R^tR)=cxx^t$ for some $c\in\mathbb{R}\setminus 0$, by \Cref{null=1.adj(S).columns} and by the proof of \Cref{G.nul=1.adj.G-v.supp.x_v^2}. Therefore, $(k-1)0=r^t(cxx^t)r=c(r^tx)(r^tx)^t$. 
		\item Since $\nul(R^tR)=\nul(R)\geq 2$, we have $\adj(R^tR)=0$ by \Cref{lemma.adj.nonzero} and $\det(R^tR)=0$. \qedhere
	\end{enumerate}
\end{proof}

%%%%%%%%%%%%%%%%%%%%%%%%%%%%%%%%%%%%%%%%%%%%%%%%%%%%%%%%%%%
%%%%%%%%%%%%%%%%%%%%%%%%%%%%%%%%%%%%%%%%%%%%%%%%%
%
%
%
%%%%%%%%%%%%%%%%%%%%%%%%%%%%%%%%%%%%%%%%%%%%%%%%%
%%%%%%%%%%%%%%%%%%%%%%%%%%%%%%%%%%%%%%%%%%%%%%%%%
%
%
%
%%%%%%%%%%%%%%%%%%%%%%%%%%%%%%%%%%%%%%%%%%%%%%%%%

%%%%%%%%%%%%%%%%%%%%%%%%%%%%%%%%%%%%%%%%%%%%%%%%%%%%%%%%%%%
\section*{Acknowledgments}
This work was partially supported by Universidad Nacional de San Luis (Argentina), grant PROICO 03-0723, Agencia I+D+i, grant PICT-2020-Serie A-00549, CONICET, grant PIP 11220220100068CO. 

The authors would like to thank the anonymous referee for their valuable comments and suggestions, which significantly improved Sections 4 and 8.

%Agencia Nacional de Promoción de la Investigación, el Desarrollo Tecnológico y la Innovación
\section*{Declaration of generative AI and AI-assisted technologies in the writing process}

During the preparation of this work, the authors used DeepSeek Chat and ChatGPT-3.5 to improve the grammatical flow and clarity of several paragraphs. Subsequent to using this service, the authors reviewed, refined, and verified the content and take full responsibility for the published work.
	
\bibliographystyle{abbrv}
\bibliography{citas__nullspace_of_split_graphs.bib}	
	
\end{document}